\documentclass[12pt]{article}
\usepackage[pdftex,bookmarksopen=true,bookmarks=true,unicode,setpagesize]{hyperref}
\hypersetup{colorlinks=true,linkcolor=black,citecolor=black}

\usepackage{amssymb, amsmath, amsthm}

\numberwithin{equation}{section}

\allowdisplaybreaks
\usepackage{cite}

\newcommand\R{\mathbb{R}}
\newcommand\N{\mathbb{N}}
\renewcommand\i{{\rm 1\kern -.3600em 1}}

\newtheorem{theorem}{Theorem}[section]
\newtheorem{corollary}[theorem]{Corollary}
\newtheorem{lemma}[theorem]{Lemma}
\newtheorem{proposition}[theorem]{Proposition}

\theoremstyle{remark}
\newtheorem{definition}[theorem]{Definition}
\theoremstyle{remark}

\theoremstyle{remark}
\newtheorem{remark}[theorem]{Remark}

\begin{document}

\vspace{-20mm}
\begin{center}{\Large \bf
An infinite dimensional umbral calculus}
\end{center}

{\large Dmitri Finkelshtein}\\ Department of Mathematics,
Swansea University, Singleton Park, Swansea SA2 8PP, U.K.;
e-mail: \texttt{d.l.finkelshtein@swansea.ac.uk}\vspace{2mm}

{\large Yuri Kondratiev}\\ Fakult\"at f\"ur Mathematik, Universit\"at Bielefeld,
            33615~Bielefeld, Germany;\\
e-mail: \texttt{kondrat@mathematik.uni-bielefeld.de}\vspace{2mm}

{\large Eugene Lytvynov}\\ Department of Mathematics,
Swansea University, Singleton Park, Swansea SA2 8PP, U.K.;
e-mail: \texttt{e.lytvynov@swansea.ac.uk}\vspace{2mm}

{\large Maria Jo\~{a}o Oliveira}\\ Departamento de Ci\^encias e Tecnologia, Universidade Aberta,
            1269-001 Lisbon, Portugal; CMAF-CIO, University of Lisbon, 1749-016 Lisbon, Portugal;\\
e-mail: \texttt{mjoliveira@ciencias.ulisboa.pt}\vspace{2mm}

{\small
\begin{center}
{\bf Abstract}
\end{center}
\noindent
The aim of this paper is to develop foundations of umbral calculus on the space $\mathcal D'$  of distributions on $\mathbb R^d$, which  leads to a general theory of Sheffer polynomial sequences on $\mathcal D'$.  We define a sequence of monic polynomials on $\mathcal D'$, a  polynomial sequence of binomial type,  and a Sheffer sequence. We present equivalent conditions  for a   sequence of monic polynomials on $\mathcal D'$  to be of binomial type or a Sheffer sequence, respectively. 
 We  also construct a  lifting of a sequence of monic polynomials on $\mathbb R$ of binomial type to a polynomial sequence of binomial type on $\mathcal D'$, and a lifting of a Sheffer sequence on $\mathbb R$ to a Sheffer sequence on $\mathcal D'$. Examples of lifted polynomial sequences  include the falling and rising factorials on $\mathcal D'$, Abel, Hermite, Charlier, and Laguerre polynomials on $\mathcal D'$. Some of these polynomials have already appeared in different branches of infinite dimensional (stochastic) analysis and played there a fundamental role.
 } \vspace{2mm}

{\bf Keywords:} Generating function; polynomial sequence on $\mathcal D'$;
polynomial sequence of binomial type on $\mathcal D'$; Sheffer sequence on $\mathcal D'$; shift-invariance; umbral calculus  on $\mathcal D'$.\vspace{2mm}

{\bf 2010 MSC. Primary:} 05A40, 46E50. {\bf Secondary:} 60H40, 60G55.

\section{Introduction}

In its modern form, umbral calculus is a study of shift-invariant linear operators acting on polynomials,  their associated polynomial sequences of binomial type, and  Sheffer sequences (including  Appell sequences). We refer to the seminal papers \cite{MR,RKO,RR}, see also the monographs \cite{KRY,Roman}. Umbral calculus has applications in combinatorics, theory of special functions, approximation theory, probability and statistics, topology, and physics, see e.g.\ the survey paper \cite{DBL}  for a long list of references.

Many extensions of umbral calculus to the case of polynomials of several, or even infinitely many variables were discussed e.g.\ in \cite{BBN,Brown,GJ,Michor,Parrish,Reiner,Roman_multivariate,SA,Ueno}, for a longer list of such papers see the introduction to \cite{DBLR}.  Appell and Sheffer sequences of polynomials of several noncommutative variables arising in the context of free probability, Boolean probability, and conditionally free probability were discussed in \cite{Anshelevich1,Anshelevich4,Anshelevich5}, see also the references therein.
 
 The paper \cite{DBLR} was a pioneering (and seemingly unique) work in which elements of basis-free umbral calculus were developed on an infinite dimensional space, more precisely, on a real separable Hilbert space $\mathcal H$.
 This paper discussed, in particular, shift-invariant linear operators acting on the space of polynomials on $\mathcal H$,  Appell sequences, and examples of polynomial sequences of binomial type.

In fact, examples of Sheffer  sequences, i.e., polynomial sequences with generating function of a certain exponential type, have appeared in infinite dimensional analysis on numerous occasions.
Some of these polynomial sequences are orthogonal with respect to a given probability measure on an infinite dimensional space, while others are related to analytical structures on such spaces.
Typically, these polynomials are either defined on a co-nuclear space $\Phi'$ (i.e, the dual of a nuclear space $\Phi$), or on an appropriate subset of $\Phi'$. Furthermore, in majority of examples, the nuclear space $\Phi$  consists of (smooth) functions on an underlying space $X$. For simplicity, we choose to work in this paper with the Gel'fand triple
$$\Phi=\mathcal D\subset L^2(\R^d,dx)\subset\mathcal D'=\Phi'.$$
Here $\mathcal D$ is the nuclear space of smooth compactly supported functions on $\R^d$, $d\in\mathbb N$, and $\mathcal D'$ is the dual space of $\mathcal D$, where the dual pairing between $\mathcal D'$ and $\mathcal D$ is obtained by continuously extending the inner product in  $ L^2(\R^d,dx)$.

 Let us mention several known examples of  Sheffer sequences on $\mathcal D'$ or its subsets:

\begin{itemize}
\item[(i)] In infinite dimensional Gaussian analysis, also called white noise analysis, Hermite polynomial sequences on $\mathcal D'$ (or rather on $ S'\subset\mathcal D'$, the Schwartz space of tempered distributions)
    appear as polynomials orthogonal with respect to Gaussian white noise measure,  see e.g.\ \cite{BK,HKPS,HOUZ,Obata}.

\item[(ii)] Charlier polynomial sequences on the configuration space of counting Radon measures on $\R^d$, $\Gamma\subset\mathcal D'$, appear as polynomials orthogonal with respect to Poisson point process on $\R^d$, see \cite{IK,Kondratievetal,KKO}.

\item[(iii)] Laguerre polynomial sequences on the cone of discrete Radon measures on $\R^d$, $\mathbb K\subset \mathcal D'$, appear as polynomials orthogonal with respect to the gamma random measure, see \cite{KL,Kondratievetal}.

\item[(iv)] Meixner polynomial sequences on $\mathcal D'$ appear as polynomials orthogonal with respect to the Meixner white noise measure, see \cite{L1,L2}.

\item[(v)] Special polynomials on the configuration space $\Gamma\subset\mathcal D'$ are used to construct the $K$-transform,  see e.g.\ \cite{BKKL,KK,KKO2}. Recall that the $K$-transform determines the duality between point processes on $\mathbb R^d$ and their correlation measures. These polynomials will be identified in this paper as the infinite dimensional analog of the falling factorials (a special case of the Newton polynomials).

\item[(vi)] Polynomial sequences on $\mathcal D'$ with generating function of a certain exponential type are used  in biorthogonal analysis related to general measures on $\mathcal D'$, see \cite{ADKS,KSWY}.

\end{itemize}

Note, however, that even the very notion of a general polynomial sequence on an infinite dimensional space has never been discussed!

The classical umbral calculus on the real line gives a general theory of Sheffer sequences
 and related (umbral) operators. So our aim in this paper is to develop foundations of umbral calculus on the space $\mathcal D'$, which will eventually lead to a general theory of Sheffer sequences on $\mathcal D'$ and their umbral operators. In fact, at a structural level, many results of this paper will have remarkable similarities to the classical  setting of polynomials on $\R$. For example,
the form of the generating function of a Sheffer sequence on $\mathcal D'$ will be similar to the generating function of a Sheffer sequence on $\mathbb R$:  the constants appearing in the latter function  are replaced in the former function by appropriate linear continuous operators.

There is a principal point in our approach that we would like to stress. The paper \cite{DBLR} deals with polynomials on a general Hilbert space $\mathcal H$, while the monograph \cite{BK} develops Gaussian analysis on a general co-nuclear space $\Phi'$, without the assumption that $\Phi'$ consists of generalized functions on $\R^d$ (or on a general underlying space).  In fact, we will discuss in Remark \ref{Hermite} below that the case of the infinite dimensional Hermite polynomials is, in a sense,  exceptional and  does not require from the co-nuclear space $\Phi'$ any special structure. In all other cases, the choice $\Phi'=\mathcal D'$  is crucial.  Having said this, let us note that our ansatz can still be applied to a rather general co-nuclear space of generalized functions over a topological space $X$, equipped with a reference measure.

We also stress that the topological aspects of the spaces $\mathcal D$, $\mathcal D'$, and their symmetric tensor powers are crucial for us since all the linear operators  (appearing as `coefficients' in our theory) are continuous in the respective topologies. Furthermore, this assumption of continuity is principal and cannot be omitted.

The origins of the classical umbral calculus are in combinatorics. So, by analogy, one can think of umbral calculus on $\mathcal D'$ as a kind of  spatial combinatorics.
To give the reader a better feeling of this, let us consider the following example. Let $\gamma=\sum_{i=1}^\infty \delta_{x_i}\in\Gamma$ be a configuration. Here $\delta_{x_i}$ denotes the Dirac measure with mass at $x_i$. We will construct  (the kernel of) the falling factorial, denoted by $(\gamma)_n$, as a function from $\Gamma$ to ${\mathcal D'}^{\odot n}$. (Here and below $\odot$ denotes the symmetric tensor product.)  This will allow us to define `$\gamma$ {\it choose} $n$' by $\binom \gamma n:=\frac1{n!}(\gamma)_n$. And we will get the following explicit formula which supports this term:
\begin{equation}\label{uytfr675e}
\binom \gamma n=\sum_{\{i_1,\dots,i_n\}\subset\mathbb N} \delta_{x_{i_1}}\odot \delta_{x_{i_2}}\odot\dots\odot \delta_{x_{i_n}},\end{equation}
i.e., the sum is obtained by choosing all possible $n$-point subsets from the (locally finite) set $\{x_i\}_{i\in\mathbb N}$. The latter set can be obviously identified with the  configuration $\gamma$.

The paper is organized as follows. In Section~2 we discuss preliminaries. In particular, we recall the construction of a general Gel'fand triple $\Phi\subset\mathcal H_0\subset\Phi'$, where $\Phi$ is a nuclear space and $\Phi'$ is the dual of $\Phi$ (a co-nuclear space) with respect to the center Hilbert space $\mathcal H_0$.   We  consider the space $\mathcal P(\Phi')$ of polynomials on $\Phi'$ and equip it  with a nuclear space topology. Its dual space, denoted by $\mathcal F(\Phi')$,  has a natural (commutative) algebraic structure with respect to the symmetric tensor product.  We  also define a family of shift operators, $(E(\zeta))_{\zeta\in\Phi'}$, and the space of shift-invariant continuous linear operators on $\mathcal P(\Phi')$, denoted by $\mathbb S(\mathcal P(\Phi'))$.

In Section 3, we give the definitions of a
polynomial sequence on $\Phi'$,  a monic polynomial sequence on $\Phi'$, and  a monic polynomial sequence on $\Phi'$ of binomial type.

Starting from Section 4, we choose the Gel'fand triple as $\mathcal D\subset L^2(\R^d,dx)\subset\mathcal D'$.  The main result of this section is Theorem~\ref{Theorem1}, which gives 
  three equivalent conditions for a monic polynomial sequence to be of binomial type. The first equivalent condition is that  the corresponding lowering operators  are shift-invariant. The second condition   gives a representation of each lowering operator through directional derivatives in directions of delta functions, $\delta_x$ ($x\in\R^d$). The third condition gives the form of the generating function of a polynomial sequence of binomial type.

To prove Theorem~\ref{Theorem1}, we derive two essential results. The first one is an operator expansion theorem (Theorem~\ref{Corollary1}), which gives a description of any shift-invariant operator $T\in\mathbb S(\mathcal P(\mathcal D'))$ in terms of the lowering operators in directions $\delta_x$. The second result is an  isomorphism theorem (Theorem~\ref{Theorem2}): we construct a bijection $J$ between the spaces  $\mathbb S(\mathcal P(\mathcal D'))$ and $\mathcal F(\mathcal D')$ such that, under $J$, the product of any shift-invariant operators goes over into the symmetric tensor product of their images.   This implies, in particular, that any two shift-invariant operators commute.

Next, we define a family of delta operators on $\mathcal D'$ and prove that, for each such family, there exists a unique monic polynomial sequence of binomial type for which these delta operators are the lowering operators.

In Section 5, we identify a procedure of the lifting of a polynomial sequence of binomial type on $\R$ to  a polynomial sequence of binomial type on $\mathcal D'$. 
This becomes possible due to the structural similarities between the one-dimensional and infinite-dimensional theories.
Using this procedure, we identify, on $\mathcal D'$, the falling factorials, the rising factorials, the Abel polynomials, and the Laguerre polynomials of binomial type. We stress that the polynomial sequences lifted from $\R$ to $\mathcal D'$ form a subset of a (much larger) set of all polynomial sequences of binomial type on $\mathcal D'$.

In Section 6, we define a Sheffer sequence on $\mathcal D'$ as a monic polynomial sequence on $\mathcal D'$ whose lowering operators are delta operators. Thus, to every Sheffer sequence, there corresponds a (unique) polynomial sequence of binomial type. In particular, if the corresponding binomial sequence is just the set of  monomials (i.e., their delta operators are differential operators), we call such a Sheffer sequence an Appell sequence.
The  main result of this section,  Theorem~\ref{Theorem3}, gives several equivalent conditions for a monic polynomial sequence to be a Sheffer sequence. In particular,  we find the generating function of a Sheffer sequence on $\mathcal D'$.

In Section 7, we extend the procedure of the lifting described in Section 5 to Sheffer sequences. Thus, for each Sheffer sequence on $\R$, we define a Sheffer sequence on $\mathcal D'$. Using this procedure, we recover, in particular, the Hermite polynomials, the Charlier polynomials,  and the orthogonal Laguerre polynomials on $\mathcal D'$.

Finally, in Appendix, we discuss  several properties of formal  tensor power series. 

From the technical point of view, the similarities between the infinite dimensional and the classical settings open new perspectives in infinite dimensional analysis. Due to the special character of this approach, namely, through definition of umbral operators on $\mathcal P(\mathcal D')$ and umbral composition of polynomials on $\mathcal D'$, further developments and applications in infinite dimensional analysis are subject of forthcoming publications.

Let us also mention the open problem of (at least partial) characterization of Sheffer sequences that are orthogonal with respect to a certain probability measure on $\mathcal D'$. In the one-dimensional case, such a characterization is due to Meixner  \cite{Meixner}.  For multi-dimensional extensions of this result, see \cite{Casalis,Pommeret1,Pommeret2} and the references therein. 

\section{Preliminaries}\label{Section2}

\subsection{Nuclear and co-nuclear spaces}

Let us first recall the definition of a nuclear space, for details see e.g.\ \cite[Chapter 14, Section~2.2]{BSU2}. Consider a family of real separable Hilbert spaces $(\mathcal H_\tau)_{\tau\in T}$, where $T$ is an arbitrary indexing set. Assume that the set $\Phi:=\bigcap_{\tau\in T}\mathcal H_\tau$ is dense in each Hilbert space $\mathcal H_\tau$ and the family  $(\mathcal H_\tau)_{\tau\in T}$ is directed by embedding, i.e., for any $\tau_1,\tau_2\in T$ there exists a $\tau_3\in T$ such that $\mathcal H_{\tau_3}\subset \mathcal H_{\tau_1}$ and $\mathcal H_{\tau_3}\subset \mathcal H_{\tau_2}$ and both embeddings are continuous. We introduce in $\Phi$ the projective limit topology of the $\mathcal H_\tau$ spaces:
$$\Phi=\operatornamewithlimits{proj\,lim}_{\tau\in T}\mathcal H_\tau.$$
By definition, the sets $\{\varphi\in\Phi\mid \| \varphi-\psi\|_{\mathcal H_\tau}<\varepsilon\}$ with $\psi\in\Phi$, $\tau\in T$, and $\varepsilon>0$ form a system of base neighborhoods in this topology. Here $\|\cdot\|_{\mathcal H_\tau}$ denotes the norm in $\mathcal H_\tau$.

Assume that, for each $\tau_1\in T$, there exists a $\tau_2\in T$ such that $\mathcal H_{\tau_2}\subset \mathcal H_{\tau_1}$, and the operator of embedding of $\mathcal H_{\tau_2}$ into $\mathcal H_{\tau_1}$ is of the Hilbert--Schmidt class. Then the linear topological space $\Phi$ is called {\it nuclear}.

Next, let us assume that, for some $\tau_0\in T$, each Hilbert space $\mathcal H_\tau$ with $\tau\in T$ is continuously embedded into $\mathcal H_0:=\mathcal H_{\tau_0}$. We will call $\mathcal H_0$ the {\it center space}.

 Let $\Phi'$ denote the dual space of $\Phi$ with respect to the center space  $\mathcal H_0$, i.e., the dual pairing between $\Phi'$ and $\Phi$ is obtained by continuously extending the inner product in  $\mathcal H_0$, see  e.g.\ \cite[Chapter 14, Section~2.3]{BSU2}. The space $\Phi'$ is often called {\it co-nuclear}.

By the Schwartz theorem (e.g.\ \cite[Chapter 14, Theorem~2.1]{BSU2}), $\Phi'=\bigcup_{\tau\in T}\mathcal H_{-\tau}$, where $\mathcal H_{-\tau}$ denotes the dual space of $\mathcal H_\tau$ with respect to the center space $\mathcal H_0$. We endow $\Phi'$ with the Mackey topology---the strongest topology in $\Phi'$ consistent with the duality between $\Phi$ and $\Phi'$ (i.e., the set of continuous linear  functionals on $\Phi'$ coincides with $\Phi$). The Mackey topology in $\Phi'$ coincides with the topology of the inductive limit of the $\mathcal H_{-\tau}$ spaces, see e.g.\ \cite[Chapter IV, Proposition 4.4]{SW} or \cite[Chapter 1, Section 1]{BK}.
Thus, we obtain   the {\it Gel'fand triple} (also called the {\it standard triple})
\begin{equation}\label{fd6re7i}
\Phi=\operatornamewithlimits{proj\,lim}_{\tau\in T} \mathcal H_\tau\subset\mathcal  H_0\subset \operatornamewithlimits{ind\,lim}_{\tau\in T}\mathcal H_{-\tau}=\Phi'.\end{equation}

Let $X$ and $Y$ be linear topological spaces that are locally convex and Hausdorff. (Both $\Phi$ and $\Phi'$ are such spaces.)
We denote by $\mathcal L(X,Y)$ the space of continuous linear operators acting
from $X$ into $Y$. We will also denote $\mathcal L(X):=\mathcal L(X,X)$.  We denote by $X'$ and $Y'$ the dual space of $X$ and $Y$, respectively. We endow $X'$ with the Mackey topology with respect to the duality between $X$ and $X'$. We similarly endow $Y'$ with the Mackey topology.

Each operator $A\in\mathcal L(X,Y)$ has the {\it adjoint operator} $A^*\in\mathcal L(Y',X')$ (also called the transpose of $A$ or the dual of $A$), see e.g.\ \cite[Theorem 8.11.3]{NB}.

\begin{remark}\label{furt86o}
Note that, since we chose the Mackey topology on $\Phi'$, for an operator $A\in\mathcal L(\Phi')$, we have $A^*\in\mathcal L(\Phi)$. This fact will be used throughout the paper.

\end{remark}

\begin{proposition}\label{cyr768o65}
 Consider the Gel'fand triple \eqref{fd6re7i}. Let $A:\Phi\to\Phi$ and $B:\Phi'\to\Phi'$ be linear operators.

{\rm (i)} We have $A\in\mathcal L(\Phi)$ if and only if, for each $\tau_1\in T$, there exists a $\tau_2\in T$ such that the operator $A$ can be extended by continuity to an operator $\hat A\in\mathcal L(\mathcal H_{\tau_2},\mathcal H_{\tau_1})$.

{\rm (ii)} We have $B\in\mathcal L(\Phi')$ if and only if, for each $\tau_1\in T$, there exists a $\tau_2\in T$ such that
the operator $\hat B:=B\restriction \mathcal H_{-\tau_1}$
takes on values in $\mathcal H_{-\tau_2}$ and $\hat B\in\mathcal L(\mathcal H_{-\tau_1},\mathcal H_{-\tau_2})$.
\end{proposition}

\begin{remark}\label{tfur7}
Proposition~\ref{cyr768o65} admits a sraightforward generalization to the case of two Gel'fand triples, $\Phi\subset \mathcal H_0\subset \Phi'$ and $\Psi\subset \mathcal G_0\subset\Psi'$, and linear operators $A:\Phi\to\Psi$ and $B:\Phi'\to\Psi'$.
\end{remark}

\begin{remark}
Part (ii) of Proposition~\ref{cyr768o65}  is related to the universal property of an inductive limit, which states that any linear operator from an inductive limit of a family of locally convex  spaces to another locally convex space is continuous if and only if the restriction of the operator to any member of the family is continuous, see e.g.\ \cite[II.29]{Bourbaki}.
\end{remark}

\begin{proof}[Proof of Proposition~\ref{cyr768o65}]
(i) By the definition of the topology in $\Phi$, the linear operator $A:\Phi\to\Phi$ is continuous if and only if, for any $\tau_1\in T$ and $\varepsilon_1>0$, there exist $\tau_2\in T$ and $\varepsilon_2>0$ such that the pre-image of the set
$$\{\varphi\in\Phi\mid \|\varphi\|_{\mathcal H_{\tau_1}}<\varepsilon_1\}$$ contains the set
$$\{\varphi\in\Phi\mid \|\varphi\|_{\mathcal H_{\tau_2}}<\varepsilon_2\}.$$
But this implies the statement.

(ii) Assume $B\in\mathcal L(\Phi')$. Then, by Remark \ref{furt86o}, we have $B^*\in\mathcal L(\Phi)$. Hence,
for each $\tau_1\in T$, there exists a $\tau_2\in T$ such that the operator $B^*$ can be extended by continuity to an operator $\hat B^*\in\mathcal L(\mathcal H_{\tau_2},\mathcal H_{\tau_1})$. But the adjoint of the operator  $\hat B^*$ is $\hat B:=B\restriction\mathcal H_{-\tau_1}$. Hence $\hat B\in\mathcal L(\mathcal H_{-\tau_1},\mathcal H_{-\tau_2})$.

Conversely, assume that, for each $\tau_1\in T$, there exists a $\tau_2\in T$ such that
the operator $\hat B:=B\restriction \mathcal H_{-\tau_1}$
takes on values in $\mathcal H_{-\tau_2}$ and $\hat B\in\mathcal L(\mathcal H_{-\tau_1},\mathcal H_{-\tau_2})$.
Therefore, $\hat B^*\in\mathcal L(\mathcal H_{\tau_2},\mathcal H_{\tau_1})$. Denote $A:=\hat B^*\restriction \Phi$. As easily seen, the definition of the operator $A$ does not depend on the choice of $\tau_1,\tau_2\in T$. Hence, $A:\Phi\to\Phi$, and by part (i) we conclude that $A\in\mathcal L(\Phi)$. But $B=A^*$ and hence $B\in\mathcal L(\Phi')$.
\end{proof}

In what follows, $\otimes$ will denote the tensor product. In particular, for a real separable Hilbert space $\mathcal H$, $\mathcal H^{\otimes n}$ denotes the $n$th tensor power of $\mathcal H$. We will denote by $\operatorname{Sym}_n\in\mathcal L(\mathcal H^{\otimes n})$ the symmetrization operator, i.e., the orthogonal projection satisfying
\begin{equation}\label{cyrd6ue7r}
\operatorname{Sym}_n f_1\otimes f_2\otimes\dots\otimes f_n=\frac1{n!}\sum_{\sigma\in\mathfrak S(n)}f_{\sigma(1)}\otimes f_{\sigma(2)}\otimes\dots\otimes f_{\sigma(n)} \end{equation}
for $f_1,f_2,\dots,f_n\in\mathcal H $.
 Here $\mathfrak S(n)$
denotes the symmetric group acting on $\{1,\dots,n\}$.
We will denote the symmetric tensor product by $\odot$. In particular,
$$f_1\odot f_2\odot\dots\odot f_n:= \operatorname{Sym}_n f_1\otimes f_2\otimes\dots\otimes f_n,\quad f_1,f_2,\dots,f_n\in\mathcal H,$$
and $\mathcal H^{\odot n}:=\operatorname{Sym}_n \mathcal H^{\otimes n}$ is the $n$th symmetric tensor power of $\mathcal H$. Note that, for each $f\in\mathcal H$, we have $f^{\odot n}=f^{\otimes n}$.

Starting with Gel'fand triple \eqref{fd6re7i}, one constructs its $n$th symmetric tensor power as follows:
$$\Phi^{\odot n}:=\operatornamewithlimits{proj\,lim}_{\tau\in T}\mathcal H_\tau^{\odot n}\subset \mathcal H_0^{\odot n}\subset \operatornamewithlimits{ind\,lim}_{\tau\in T}\mathcal H_{-\tau}^{\odot n}=:\Phi'{}^{\odot n},$$
see e.g.\ \cite[Section~2.1]{BK} for details. In particular, $\Phi^{\odot n}$ is a nuclear space and $\Phi'{}^{\odot n}$ is its dual with respect to the center space $\mathcal H_0^{\odot n}$.
We will also denote $\Phi^{\odot 0}=\mathcal H_0^{\odot 0}=\Phi'{}^{\odot 0}:=\R$. The dual pairing between $F^{(n)}\in \Phi'{}^{\odot n}$ and $g^{(n)}\in\Phi^{\odot n}$ will be denoted by  $\langle F^{(n)},g^{(n)}\rangle$.

\begin{remark}\label{t5768909} Consider the set $\{\xi^{\otimes n}\mid \xi\in\Phi\}$. By the polarization identity, the linear span of this set is dense in every space $\mathcal H_\tau^{\odot n}$, $\tau\in T$.
\end{remark}

The following lemma will be very important for our considerations.

\begin{lemma}\label{yde6ue6cb}
 {\rm (i)} Let $F^{(n)},G^{(n)}\in\Phi'{}^{\odot n}$ be such that
$$\langle F^{(n)},\xi^{\otimes n}\rangle=\langle G^{(n)},\xi^{\otimes n}\rangle\quad\text{for all }\xi\in\Phi,$$
then $F^{(n)}=G^{(n)}$.

{\rm (ii)} Let $\Phi$ and $\Psi$ be nuclear spaces and let $A,B\in\mathcal L(\Phi^{\odot n},\Psi)$. Assume that
$$A\xi^{\otimes n}=B\xi^{\otimes n}\quad \text{for all }\xi\in\Phi.$$ Then $A=B$.
\end{lemma}

\begin{proof}
Statement (i) follows from  Remark \ref{t5768909}, statement (ii) follows from Proposition~\ref{cyr768o65}, (i) and Remarks \ref{tfur7} and \ref{t5768909}.
\end{proof}

\subsection{Polynomials on a co-nuclear space}

Below we fix the Gel'fand triple \eqref{fd6re7i}.

\begin{definition}
A function $P:\Phi'\to\R$ is called a {\it polynomial on $\Phi'$} if
\begin{equation}\label{dtdsr6e}
P(\omega)=\sum_{k=0}^n\langle\omega^{\otimes k},f^{(k)}\rangle,\quad \omega\in\Phi',\end{equation}
where $f^{(k)}\in\Phi^{\odot k}$, $k=0,1,\dots,n$, $n\in\mathbb N_0:=\{0,1,2,\dots\}$, and $\omega^{\otimes 0}:=1$. If $f^{(n)}\ne 0$, one says that the {\it polynomial $P$ is of degree $n$}.
We denote by $\mathcal P(\Phi')$ the set of all polynomials on $\Phi'$.
\end{definition}

\begin{remark}
For each $P\in\mathcal P(\Phi')$, its representation in  form \eqref{dtdsr6e} is evidently unique.
\end{remark}

For any $f^{(k)}\in\Phi^{\odot k}$ and $g^{(n)}\in\Phi^{\odot n}$, $k,n\in\mathbb N_0$, we have
\begin{equation}\label{897y6t}
\langle\omega^{\otimes k},f^{(k)}\rangle \langle\omega^{\otimes n},g^{(n)}\rangle=\langle \omega^{\otimes (k+n)},f^{(k)}\odot g^{(n)}\rangle,\quad \omega\in\Phi'.\end{equation}
Hence $\mathcal P(\Phi')$ is an algebra under point-wise multiplication of polynomials on $\Phi'$.

We will now define a topology on $\mathcal P(\Phi')$. Let $\mathcal{F}_{\mathrm{fin}}(\Phi)$ denote the topological direct sum of the nuclear spaces $\Phi^{\odot n}$, $n\in\mathbb N_0$. Hence, $\mathcal{F}_{\mathrm{fin}}(\Phi)$ is a nuclear space, see e.g.\ \cite[Section~5.1]{BK}. This space consists of all finite sequences $f=(f^{(0)},f^{(1)},\dots,f^{(n)},0,0,\dots)$, where $f^{(k)}\in\Phi^{\odot k}$, $k=0,1,\dots,n$, $n\in\mathbb N_0$. The convergence in $\mathcal{F}_{\mathrm{fin}}(\Phi)$ means the uniform finiteness of non-zero elements and the coordinate-wise convergence in each $\Phi^{\odot k}$.

\begin{remark}Below we will often identify $f^{(n)}\in\Phi^{\odot n}$ with
$$(0,\dots,0,f^{(n)},0,0,\dots)\in\mathcal{F}_{\mathrm{fin}}(\Phi).$$
\end{remark}

We define a natural bijective mapping
$I:\mathcal{F}_{\mathrm{fin}}(\Phi)\to\mathcal{P}(\Phi')$
by
\begin{equation}
(If)(\omega):=\sum_{k=0}^n \langle\omega^{\otimes k},f^{(k)}\rangle,\label{Eq5}
\end{equation}
for $f=(f^{(0)},f^{(1)},\dots,f^{(n)},0,0,\dots)\in \mathcal{F}_{\mathrm{fin}}(\Phi)$. We  define a nuclear space topology on  $\mathcal{P}(\Phi')$ as the image of the topology on  $\mathcal{F}_{\mathrm{fin}}(\Phi)$ under the mapping $I$.

The  space $\mathcal F_{\mathrm{fin}}(\Phi)$ may be endowed with the structure of an algebra with respect to the symmetric tensor product
\begin{equation}\label{prodonFfin}
            f\odot g =\left( \sum_{i=0}^k f^{(i)}\odot g^{(k-i)}\right)_{k=0}^\infty,
          \end{equation}
          where $f=(f^{(k)})_{k=0}^\infty, \,
          g=(g^{(k)})_{k=0}^\infty\in\mathcal F_{\mathrm{fin}}(\Phi)$.
           The unit element of this (commutative) algebra is the {\em vacuum vector} $\Omega:=(1,0,0\ldots)$.

          By \eqref{897y6t}--\eqref{prodonFfin},
          the bijective mapping $I$  provides an isomorphism between the algebras $\mathcal{F}_{\mathrm{fin}}(\Phi)$ and $\mathcal P(\Phi')$, namely, for any $f,g\in\mathcal{F}_{\mathrm{fin}}(\Phi)$,
          $$
            \bigl(I(f\odot g)\bigr)(\omega)=(If)(\omega) (Ig)(\omega), \quad \omega\in\Phi'.
          $$

Let
$$\mathcal F(\Phi'):=\prod_{k=0}^\infty \Phi'{}^{\odot k}$$
denote the topological product of the spaces $\Phi'{}^{\odot k}$. The space $\mathcal F(\Phi')$ consists of all sequences $F=(F^{(k)})_{k=0}^\infty$, where $F^{(k)}\in \Phi'{}^{\odot k}$, $k\in\mathbb N_0$. Note that the convergence in this space means the coordinate-wise convergence in each space $ \Phi'{}^{\odot k}$.

Each element $F=(F^{(k)})_{k=0}^\infty\in \mathcal F(\Phi')$ determines a continuous linear functional on $\mathcal{F}_{\mathrm{fin}}(\Phi)$ by
\begin{equation}\label{gfyr7r}
\langle F,f\rangle:=\sum_{k=0}^\infty \langle F^{(k)},f^{(k)}\rangle,\quad f=(f^{(k)})_{k=0}^\infty\in \mathcal{F}_{\mathrm{fin}}(\Phi)\end{equation}
(note that the sum in \eqref{gfyr7r} is, in fact, finite). The dual of $\mathcal{F}_{\mathrm{fin}}(\Phi)$ is equal to $\mathcal F(\Phi')$, and  the topology on $\mathcal F(\Phi')$ coincides with the Mackey topology on $\mathcal F(\Phi')$ that is consistent with the duality between $\mathcal{F}_{\mathrm{fin}}(\Phi)$ and $\mathcal F(\Phi')$, see e.g.\ \cite{BK}. In view of the definition of the topology on $\mathcal P(\Phi')$, we may also think of $\mathcal F(\Phi')$ as the dual space of $\mathcal P(\Phi')$.

Similarly to \eqref{prodonFfin}, one can introduce the
symmetric tensor   product on $\mathcal F(\Phi')$:
       \begin{equation}\label{prodonFDprime}
            F\odot G =\left( \sum_{i=0}^k F^{(i)}\odot G^{(k-i)}\right)_{k=0}^\infty,
          \end{equation}
          where $F=(F^{(k)})_{k=0}^\infty,\,
          G=(G^{(k)})_{k=0}^\infty\in\mathcal F(\Phi')$. The unit element of this algebra is again $\Omega=(1,0,0,\ldots)$.

We will now discuss another realization of the space   $\mathcal F(\Phi')$.
We denote by $\mathcal S(\R,\R)$ the vector space of formal series $R(t)=\sum_{n=0}^\infty r_nt^n$ in powers of $t\in \R$, where $r_n\in\R$ for $n\in\N_0$. The $\mathcal S(\R,\R)$ is an algebra under the product of formal power series. Similarly to $\mathcal S(\R,\R)$,  we give the following

\begin{definition}\label{es5w46u3} Each $(F^{(n)})_{n=0}^\infty\in\mathcal F(\Phi')$ identifies  a {\it `real-valued' formal series\linebreak $\sum_{n=0}^\infty \langle F^{(n)},\xi^{\otimes n}\rangle$ in tensor powers of $\xi\in\Phi$}.
We denote by $\mathcal S(\Phi,\R)$ the vector space of such formal series with natural operations. We define a product on $\mathcal S(\Phi,\R)$ by
\begin{equation}\label{vut9p}
\left(\sum_{n=0}^\infty \langle F^{(n)},\xi^{\otimes n}\rangle\right)\left(\sum_{n=0}^\infty \langle G^{(n)},\xi^{\otimes n}\rangle\right)=\sum_{n=0}^\infty\left\langle \sum_{i=0}^n F^{(i)}\odot G^{(n-i)},\xi^{\otimes n}\right\rangle,
\end{equation}
where $(F^{(n)})_{n=0}^\infty,(G^{(n)})_{n=0}^\infty\in\mathcal F(\Phi')$.
 \end{definition}

\begin{remark}
Assume that, for some  $(F^{(n)})_{n=0}^\infty,(G^{(n)})_{n=0}^\infty\in\mathcal F(\Phi')$ and $\xi\in\Phi$,  both series $\sum_{n=0}^\infty \langle F^{(n)},\xi^{\otimes n}\rangle$ and $\sum_{n=0}^\infty \langle G^{(n)},\xi^{\otimes n}\rangle$ converge absolutely. Then also the series on the right hand side of \eqref{vut9p} converges absolutely and \eqref{vut9p} holds as an equality of two real numbers.
\end{remark}

\begin{remark}\label{d6e4i7}
Let $t\in\R$ and $\xi\in\Phi$. Then, $t\xi\in\Phi$ and for $(F^{(n)})_{n=0}^\infty\in\mathcal F(\Phi')$,
 \begin{equation}\label{cte7869}
 \sum_{n=0}^\infty \langle F^{(n)},(t\xi)^{\otimes n}\rangle=
\sum_{n=0}^\infty t^n\langle F^{(n)},\xi^{\otimes n}\rangle,
 \end{equation}
 the expression on the right hand side of equality \eqref{cte7869}
 being the formal power series in $t$ that has coefficient $\langle F^{(n)},\xi^{\otimes n}\rangle$ by $t^n$.
\end{remark}

According to the definition of $\mathcal S(\Phi,\R)$, there exists a natural bijective mapping\linebreak $\mathcal I:\mathcal F(\Phi')\to\mathcal S(\Phi,\R) $ given by
\begin{equation}\label{ukt896p078}
(\mathcal IF)(\xi):=\sum_{n=0}^\infty \langle F^{(n)},\xi^{\otimes n}\rangle,\quad F=(F^{(n)})_{n=0}^\infty\in\mathcal F(\Phi'),\quad \xi\in\Phi.\end{equation}
The mapping $\mathcal I$ provides an isomorphism between the algebras $\mathcal F(\Phi')$ and $\mathcal S(\Phi,\R)$, namely,
for any $F,G\in \mathcal F(\Phi')$,
\begin{equation}\label{ukt896p0781}
\big(\mathcal I(F\odot G)\big)(\xi)=(\mathcal IF)(\xi)(\mathcal IG)(\xi)\end{equation}
see \eqref{prodonFDprime} and \eqref{vut9p}.

\begin{remark}
In view of the isomorphism $\mathcal I$, we may think of $\mathcal S(\Phi,\R)$ as the dual space of $\mathcal P(\Phi')$.
\end{remark}

Analogously to Definition \ref{es5w46u3} and Remark \ref{d6e4i7}, we can introduce a space of $\Phi$-valued tensor power series.

\begin{definition}\label{yr8o}
Let $( A_n)_{n=1}^\infty$ be a sequence of operators $A_n\in\mathcal L(\Phi^{\odot n},\Phi)$. Then the operators $( A_n)_{n=1}^\infty$ can be identified with a {\it `$\Phi$-valued' formal series $\sum_{n=1}^\infty  A_n \xi^{\otimes n}$ in tensor powers of $\xi\in\Phi$}. We denote by $\mathcal S(\Phi,\Phi)$ the vector space of such formal series.
\end{definition}

\begin{remark}
Let $t\in\R$ and $\xi\in\Phi$. Then, $t\xi\in\Phi$ and for
a sequence $( A_n)_{n=1}^\infty$ as in Definition \ref{yr8o}
 $$\sum_{n=1}^\infty  A_n(t\xi)^{\otimes n}=\sum_{n=1}^\infty t^n A_n\xi^{\otimes n}
 $$
 is the formal power series in $t$ that has  coefficient $A_n\xi^{\otimes n}\in\Phi$ by $t^n$.  Recall that, by Lemma~\ref
{yde6ue6cb}, (ii), the values of the operator $A_n$ on the vectors $\xi^{\otimes n}\in\Phi^{\odot n}$ uniquely identify the operator $A_n$.\end{remark}

In Appendix, we discuss several properties of formal  tensor power series.

\subsection{Shift-invariant operators}

\begin{definition}
For each  $\zeta\in\Phi '$, we define the {\it operator $D(\zeta):\mathcal P(\Phi ')\to\mathcal P(\Phi ')$ of differentiation in direction~$\zeta$} by
\[
(D(\zeta)P)(\omega):=\lim_{t\to 0}\frac{P(\omega+t\zeta)-P(\omega)}{t},\quad
P\in\mathcal{P}(\Phi '),\ \omega\in\Phi '.
\]
\end{definition}

\begin{definition}\label{tee665i4ei}
For each
$\zeta\in\Phi '$,  we define the {\it annihilation operator} $\mathfrak A(\zeta)\in\mathcal{L}(\mathcal{F}_{\mathrm{fin}}(\Phi ))$  by
$$
\mathfrak A(\zeta)\Omega:=0,\quad \mathfrak A(\zeta)\xi^{\otimes n}:=n
\langle\zeta,\xi\rangle \xi^{\otimes(n-1)}\quad\text{for }\xi\in\Phi,\ n\in\N.
$$
\end{definition}

\begin{lemma}\label{yufr7r8865}
 For each $\zeta\in\Phi '$, we have $D(\zeta)\in \mathcal{L}(\mathcal{P}(\Phi '))$.
\end{lemma}

\begin{proof}
Using the bijection $I:\mathcal{F}_{\mathrm{fin}}(\Phi )\to\mathcal{P}(\Phi ')$ defined by \eqref{Eq5},
we easily obtain
$$D(\zeta)=I\mathfrak A(\zeta)I^{-1},\quad \zeta\in\Phi ',$$
which implies the statement.
\end{proof}

\begin{definition}
For each $\zeta\in\Phi '$, we define the {\it operator $E(\zeta):\mathcal P(\Phi ')\to\mathcal P(\Phi ')$ of shift by $\zeta$} by
\[
(E(\zeta)P)(\omega):=P(\omega+\zeta),\quad P\in\mathcal{P}(\Phi '),\
\omega\in\Phi '.
\]
\end{definition}

\begin{lemma}(Boole's formula)\label{Lemma4}
For each $\zeta\in\Phi '$,
\[
E(\zeta)=\sum_{k=0}^\infty\frac{1}{k!}D(\zeta)^k.
\]
\end{lemma}

\begin{proof} Note that the infinite sum is, in fact, a finite sum
when applied to a polynomial, and thus it is a well-defined operator
on $\mathcal{P}(\Phi ')$. For each $\xi\in\Phi $ and $n\in\N$,
\begin{align*}
&\big(E(\zeta)\langle\cdot^{\otimes n},\xi^{\otimes n}\rangle\big)(\omega)
=\langle(\omega+\zeta)^{\otimes n},\xi^{\otimes n}\rangle\nonumber
=\sum_{k=0}^n\binom{n}{k}\langle\omega,\xi\rangle^k \langle\zeta,\xi\rangle^{n-k}\nonumber \\
&\quad=\sum_{k=0}^n\frac1{k!}\big(D(\zeta)^k\langle\cdot^{\otimes n},\xi^{\otimes n}\rangle\big)(\omega)=\sum_{k=0}^\infty\frac1{k!}\big(D(\zeta)^k\langle\cdot^{\otimes n},\xi^{\otimes n}\rangle\big)(\omega),
\end{align*}
which implies the statement.
\end{proof}

Note that Lemmas \ref{yufr7r8865} and \ref{Lemma4} imply that  $E(\zeta)\in\mathcal{L}(\mathcal{P}(\Phi '))$  for each $\zeta\in\Phi '$.

\begin{definition}
 We say that an  operator $T\in\mathcal{L}(\mathcal{P}(\Phi '))$ is
 {\it shift-invariant\/} if
 $$TE(\zeta)=E(\zeta)T\quad\text{for all }\zeta\in\Phi'.$$
  We denote the linear space of all shift-invariant operators  by $\mathbb{S}(\mathcal P(\Phi'))$. The space $\mathbb{S}(\mathcal P(\Phi'))$ is an algebra under  the usual product (composition) of operators.
\end{definition}

Obviously, for each $\zeta\in\Phi'$, the operators  $D(\zeta)$ and $E(\zeta)$ belong to $\mathbb{S}(\mathcal P(\Phi'))$.

\section{Monic polynomial sequences on a co-nuclear space}\label{gfu8cd}

We will now introduce the notion of a polynomial sequence on $\Phi'$.

For each $n\in\mathbb N_0$, we denote by $\mathcal P^{(n)}(\Phi ')$ the subspace of $\mathcal P(\Phi ')$ that consists of all polynomials on $\Phi '$ of degree $\le n$.

\begin{lemma}\label{frt}
A mapping $\mathbb P^{(n)}:\Phi ^{\odot n}\to\mathcal P^{(n)}(\Phi')$ is linear and continuous if and only if  it is of the form 
\begin{equation}\label{bt7r95}
\big(\mathbb P^{(n)}f^{(n)}\big)(\omega)=\langle P^{(n)}(\omega),f^{(n)}\rangle ,\end{equation}
where $P^{(n)}:\Phi'\to \Phi '{}^{\odot n}$ is a continuous mapping of the form 
\begin{equation}\label{ty65e48z}
P^{(n)}(\omega)=\sum_{k=0}^n U_{n,k}\,\omega^{\otimes k}\end{equation}
with $U_{n,k}\in\mathcal L(\Phi '{}^{\odot k},\Phi '{}^{\odot n})$.
\end{lemma}

\begin{proof}

Let  $\mathbb P^{(n)}\in\mathcal L\big(\Phi ^{\odot n},\mathcal P^{(n)}(\Phi ')\big)$. Then, by the definition of $\mathcal P^{(n)}(\Phi ')$, there exist operators $V_{k,n}\in\mathcal L(\Phi ^{\odot n},\Phi ^{\odot k})$, $k=0,1,\dots,n$, such that, for any $f^{(n)}\in\Phi ^{\odot n}$ and $\omega\in\Phi '$
\begin{align*}
\big(P^{(n)}f^{(n)}\big)(\omega)&=\sum_{k=0}^n\langle \omega^{\otimes k},V_{k,n}f^{(n)}\rangle\\
&=\sum_{k=0}^n\langle U_{n,k}\,\omega^{\otimes k},f^{(n)}\rangle,
\end{align*}
 where $U_{n,k}:=V^*_{k,n}\in\mathcal L(\Phi '{}^{\odot k},\Phi '{}^{\odot n})$. 
Conversely, every $P^{(n)}$ of the form \eqref{ty65e48z} determines $\mathbb P^{(n)}\in\mathcal L\big(\Phi ^{\odot n},\mathcal P^{(n)}(\Phi ')\big)$ by formula \eqref{bt7r95}.
 \end{proof}

\begin{definition}\label{ye6eu0y87}
Assume that, for each $n\in\mathbb N_0$,  $P^{(n)}:\Phi'\to \Phi '{}^{\odot n}$
is of the form \eqref{ty65e48z} with $U_{n,k}\in\mathcal L(\Phi '{}^{\odot k},\Phi '{}^{\odot n})$. Furthermore, assume that, for each $n\in\mathbb N_0$,   $U_{n,n}\in \mathcal L(\Phi '{}^{\odot n})$ is  a homeomorphism.
Then we call $(P^{(n)})_{n=0}^\infty$ a  {\it  polynomial sequence on $\Phi '$}.

If additionally, for each $n\in\mathbb N_0$,  $U_{n,n}=\mathbf 1$, the identity operator on
$\Phi'{}^{\odot n}$, then we call $(P^{(n)})_{n=0}^\infty$ a {\it  monic polynomial sequence on $\Phi '$}.

\end{definition}

\begin{remark}
Below, to simplify notations, we will only deal with monic polynomial sequences.  The results of this paper can be extended to the case of a general polynomial sequence on $\Phi'=\mathcal D'$.
\end{remark}

\begin{remark} By the definition of a monic polynomial sequence we get
\begin{equation}
\langle P^{(n)}(\omega),f^{(n)}\rangle=\langle\omega^{\otimes n},
f^{(n)}\rangle+\sum_{k=0}^{n-1}\langle\omega^{\otimes k},
V_{k,n}f^{(n)}\rangle,\quad f^{(n)}\in\Phi ^{\odot n}\label{gtdyf7ti}
\end{equation}
where $V_{k,n}:=U_{n,k}^*\in \mathcal L(\Phi ^{\odot n},\Phi ^{\odot k})$.
\end{remark}

\begin{lemma}\label{Lemma1}
Let $(P^{(n)})_{n=0}^\infty$ be a monic polynomial sequence on $\Phi '$.
The following statements hold.

 {\rm (i)}
There exist operators $R_{k,n}\in\mathcal{L}(\Phi^{\odot n},\Phi^{\odot k})$, $k=0,1,\dots,n-1$, $n\in\mathbb N$, such that, for all $\omega\in\Phi'$ and
$f^{(n)}\in\Phi^{\odot n}$,
\begin{equation}\label{vgftydf7}
\langle\omega^{\otimes n},f^{(n)}\rangle
=\langle P^{(n)}(\omega), f^{(n)}\rangle +\sum_{k=0}^{n-1}\langle P^{(k)}(\omega), R_{k,n}f^{(n)}\rangle.\end{equation}

{\rm (ii)} We have
\begin{equation}\label{tse546i78}
\mathcal{P}(\Phi')=\left\{\sum_{k=0}^n\langle P^{(k)},f^{(k)}\rangle\mid
f^{(k)}\in\Phi^{\odot k},\ k=0,1,\dots,n,\ n\in\N_0\right\}.
\end{equation}
\end{lemma}

\begin{proof} (i) We prove by induction on $n$. For $n=1$, the statement trivially holds. Assume that the statement holds for $1,2,\dots,n$. Then, by using \eqref{gtdyf7ti} and the induction assumption, we get
\begin{align*}
&\langle\omega^{\otimes (n+1)},f^{(n+1)}\rangle=\langle P^{(n+1)}(\omega),f^{(n+1)}\rangle-\sum_{k=0}^n\langle \omega^{\otimes k},V_{k,n+1}f^{(n+1)}\rangle\\
&=\langle P^{(n+1)}(\omega),f^{(n+1)}\rangle-\sum_{k=0}^n\left(
\langle P^{(k)}(\omega),V_{k,n+1}f^{(n+1)}\rangle+\sum_{i=0}^{k-1}\langle P^{(i)}(\omega),R_{i,k}V_{k,n+1}f^{(n+1)}\rangle
\right),
\end{align*}
which implies the statement for $n+1$.

(ii) This follows immediately from (i).
\end{proof}

\begin{definition}
Let $(P^{(n)})_{n=0}^\infty$ be a monic polynomial sequence on $\Phi '$.
For each $\zeta\in\Phi '$, we define a {\it lowering operator} $Q(\zeta)$  as the linear operator on $\mathcal P(\Phi ')$ (cf.\ \eqref{tse546i78}) satisfying
\begin{align*}
&Q(\zeta)\langle P^{(n)},f^{(n)}\rangle:=
\langle P^{(n-1)},\mathfrak A(\zeta)f^{(n)}\rangle,\quad f^{(n)}\in\Phi^{\odot n},\ n\in\N,\\
&Q(\zeta)\langle P^{(0)},f^{(0)}\rangle:=0,\quad f^{(0)}\in\R,
\end{align*}
where the operator $\mathfrak A(\zeta)$ is defined by Definition \ref{tee665i4ei}. \end{definition}

\begin{lemma}\label{Lemma3}
For every $\zeta\in\Phi'$, we have
$Q(\zeta)\in\mathcal{L}(\mathcal{P}(\Phi'))$.
\end{lemma}

\begin{proof} We define an operator
$R\in\mathcal{L}(\mathcal{F}_{\mathrm{fin}}(\Phi))$ by setting, for each
$f^{(n)}\in\Phi^{\odot n}$,
\[
\big(Rf^{(n)}\big){}^{(k)}:=\begin{cases}
R_{k,n}f^{(n)},& k< n\\
f^{(n)},& k=n,\\
0,& k> n,
\end{cases}
\]
where the operators $R_{k,n}$ are as in \eqref{vgftydf7}. Similarly, using the operators $V_{k,n}$ from formula \eqref{gtdyf7ti}, we define an operator $V\in\mathcal{L}(\mathcal{F}_{\mathrm{fin}}(\Phi))$.
As easily seen,
\[
Q(\zeta)=IV\mathfrak A(\zeta)RI^{-1},
\]
where $I$ is the homeomorphism defined by \eqref{Eq5}. This  implies the
required result.
\end{proof}

The simplest example of a monic polynomial sequence  on $\Phi'$ is $P^{(n)}(\omega)=\omega^{\otimes n}$, $n\in\mathbb N_0$. In this case, for $f^{(n)}\in\Phi^{\odot n}$,
$\langle P^{(n)}(\omega),f^{(n)}\rangle=\langle \omega^{\otimes n},f^{(n)}\rangle$ is just a {\it monomial on $\Phi'$ of degree $n$}.
For each $\zeta\in\Phi'$, we obviously have $Q(\zeta)=D(\zeta)$, i.e., the corresponding lowering operators are just differentiation operators.
 Furthermore, we trivially see in this case that,
 for any $n\in\mathbb N$ and any $\omega,\zeta\in\Phi'$,
\begin{equation}
P^{(n)}(\omega+\zeta)=\sum_{k=0}^n\binom{n}{k}P^{(k)}(\omega)\odot P^{(n-k)}(\zeta).\label{Eq2}
\end{equation}

\begin{definition} Let $(P^{(n)})_{n=0}^\infty$ be a monic polynomial sequence on $\Phi'$.
We say that  $(P^{(n)})_{n=0}^\infty$ is of {\it binomial type} if, for any $n\in\mathbb N$ and any $\omega,\zeta\in\Phi'$, formula \eqref{Eq2} holds.
\end{definition}

\begin{remark}\label{Remark1}
A monic polynomial sequence  $(P^{(n)})_{n=0}^\infty$  is of  binomial type if and only if,  for any $n\in\mathbb N$, $\omega,\zeta\in\Phi'$, and $\xi\in\Phi$,
\[
\langle P^{(n)}(\omega+\zeta),\xi^{\otimes n}\rangle=\sum_{k=0}^n\binom{n}{k}\langle P^{(k)}(\omega),\xi^{\otimes k}\rangle \langle P^{(n-k)}(\zeta),\xi^{\otimes (n-k)}\rangle.
\]
\end{remark}

The following lemma will be important for our considerations.

\begin{lemma}\label{Lemma2}
Let $(P^{(n)})_{n=0}^\infty$ be a monic polynomial sequence  on
$\Phi'$ of binomial type. Then, for each $n\in\N$, $P^{(n)}(0)=0$.
\end{lemma}

\begin{proof}
We proceed by induction on $n$. For $n=1$, it follows from \eqref{Eq2} that
\[
P^{(1)}(\omega+\zeta)=P^{(1)}(\omega)+ P^{(1)}(\zeta),\quad\omega,\zeta\in\Phi'.
\]
Setting $\zeta=0$, one obtains $P^{(1)}(0)=0$. Assume that the statement
holds for $1,2,\dots n$. Then, for all $\omega,\zeta\in\Phi'$,
\[
P^{(n+1)}(\omega+\zeta)=P^{(n+1)}(\omega)+ \sum_{k=1}^n\binom{n+1}{k}P^{(k)}(\omega)\odot P^{(n+1-k)}(\zeta)+P^{(n+1)}(\zeta).
\]
Setting $\zeta=0$, we conclude   $P^{(n+1)}(0)=0$.
\end{proof}

\section{Equivalent characterizations of a polynomial\\ sequence of binomial type}\label{f8oqtf}

Our next aim is to derive  equivalent characterizations of a polynomial sequence on $\Phi'$ of binomial type. As mentioned in Introduction, it will  be important for our considerations that $\Phi'$ will be chosen as a space of generalized functions.

So we fix $d\in\mathbb N$ and choose $\Phi$ to be the nuclear space $\mathcal{D}:=C_0^\infty(\R^d)$  of all
real-valued smooth functions on $\R^d$ with compact support.
 More precisely, let $T$ denote the set of all pairs $(l,\varphi)$ with $l\in\mathbb N_0$ and $\varphi\in C^\infty(\R^d)$, $\varphi(x)\ge1$ for all $x\in\R^d$. For each $\tau=(l,\varphi)\in T$, we denote by $\mathcal H_\tau$ the Sobolev  space $W^{l,2}(\mathbb R^d, \varphi(x)\,dx)$. Then
$$\mathcal D=\operatornamewithlimits{proj\,lim}_{\tau\in T}\mathcal H_\tau,$$
 see \cite[Chapter 14, Subsec.~4.3]{BSU2} for details. As the center space $\mathcal H_0=\mathcal H_{\tau_0}$ we choose $L^2(\R^d,dx)$ (i.e., $\tau_0=(0,1)$). Thus, we obtain
 the Gel'fand triple
\[
\mathcal{D}\subset L^2(\R^d,dx)\subset \mathcal{D}'.
\]

Note that nuclear space $\mathcal D^{\odot n}$ consists of all functions from  $C^\infty_0((\R^d)^n)$ that are symmetric in the variables $(x_1,\dots,x_n)\in(\R^d)^n$.

For each $x\in\R^d$, the delta function $\delta_x$ belongs to $\mathcal D'$, and we will use the notations
$$D(x):=D(\delta_x),\quad E(x):=E(\delta_x),\quad Q(x):=Q(\delta_x)$$
(the latter operator being defined for a given fixed monic polynomial sequence on $\mathcal D'$).

Below, for  any
$F^{(k)}\in\mathcal{D}'{}^{\odot k}$ and
$f^{(k)}\in\mathcal{D}^{\odot k}$, $k\in\mathbb N_0$, we denote
\[
\langle F^{(k)}(x_1,\ldots,x_k),f^{(k)}(x_1,\ldots,x_k)\rangle
:=\langle F^{(k)},f^{(k)}\rangle.
\]

\begin{theorem}\label{Theorem1}
Let $(P^{(n)})_{n=0}^\infty$ be a monic polynomial sequence
 on
$\mathcal D'$ such that $P^{(n)}(0)=0$ for all $n\in\N$. Let
$(Q(\zeta))_{\zeta\in\mathcal{D}'}$ be the corresponding  lowering operators.
Then the following conditions are equivalent:
\begin{enumerate}

\item[{\rm (BT1)}] The sequence $(P^{(n)})_{n=0}^\infty$ is of binomial type.

\item[{\rm (BT2)}] For each $\zeta\in\mathcal{D}'$, $Q(\zeta)$ is shift-invariant.

\item[{\rm (BT3)}] There exists  a sequence $(B_k)_{k=1}^\infty$ with
$B_k\in\mathcal{L}(\mathcal{D}',\mathcal{D}'^{\odot  k})$,
$k\geq 2$ and  $B_1=\mathbf 1$, the identity operator on  $\mathcal D'$, such that for all $\zeta\in\mathcal{D}'$
and  $P\in\mathcal{P}(\mathcal{D}')$,
\begin{equation}\label{tfer67}
(Q(\zeta)P)(\omega)=\sum_{k=1}^\infty\frac{1}{k!}
\big\langle (B_k\zeta)(x_1,\ldots,x_k), (D(x_1)\dotsm D(x_k)P)(\omega)\big\rangle,
\quad\omega\in\mathcal{D}'.
\end{equation}

\item[{\rm (BT4)}] The monic polynomial sequence $(P^{(n)})_{n=0}^\infty$ has the generating function
\begin{align}
\sum_{n=0}^\infty\frac{1}{n!}\langle P^{(n)}(\omega),\xi^{\otimes n}\rangle&=
\sum_{n=0}^\infty\frac{1}{n!}\langle\omega^{\otimes n},A(\xi)^{\otimes n}\rangle\notag\\
&=\sum_{n=0}^\infty\frac{1}{n!}\left\langle \omega, A(\xi)\right\rangle^n=\exp\big[\langle\omega,A(\xi)\rangle\big],\quad\omega\in\mathcal D'.\label{Eq23}
\end{align}
Here
\begin{equation}\label{ft7r}
A(\xi)=\sum_{k=1}^\infty A_k\xi^{\otimes k}\in
\mathcal S(\mathcal D,\mathcal D),\end{equation}
where
$A_k\in\mathcal{L}(\mathcal{D}^{\odot  k},\mathcal{D})$, $k\geq 2$, and
$A_1=\mathbf 1$, the identity operator on  $\mathcal D$, while \eqref{Eq23} is an equality in $\mathcal S(\mathcal D,\R)$.
\end{enumerate}
\end{theorem}

\begin{remark}
Note that $\sum_{n=0}^\infty\frac{1}{n!}\langle\omega^{\otimes n},A(\xi)^{\otimes n}\rangle\in\mathcal S(\mathcal D,\R)$ in  formula \eqref{Eq23} is the composition of $\exp[\langle\omega,\xi\rangle]:=\sum_{n=0}^\infty\frac1{n!}\,\langle \omega^{\otimes n},\xi^{\otimes n}\rangle\in\mathcal S(\mathcal D,\R)$ and $A(\xi)\in \mathcal S(\mathcal D,\mathcal D)$.  \end{remark}

\begin{remark}\label{bhgyutfu} Let $A(\xi)\in\mathcal S(\mathcal D,\mathcal D)$ be as in \eqref{Eq23}. Denote
$$B(\xi):=\sum_{k=1}^\infty \frac1{k!}\,B^*_k\xi^{\otimes k}\in\mathcal S(\mathcal D,\mathcal D),$$
where the operators $B_k$ are as in (BT3). It will follow from the proof of Theorem~\ref{Theorem1} that $B(\xi)$ is the compositional inverse of $A(\xi)$, see Definition~\ref{vtre7i}, Proposition~\ref{rw4w4w5}, and Remark~\ref{yrei4vvv}.
\end{remark}

\begin{remark}
It will also follow from the proof of Theorem~\ref{Theorem1} that, in (BT3), for each $k\ge2$, we have $B_k=R_{1,k}^*$, the adjoint of the operator $R_{1,k}\in\mathcal L(\mathcal D^{\odot k},\mathcal D)$ from formula \eqref{vgftydf7}.
 \end{remark}

Before proving this theorem, let us first note its immediate corollary.

\begin{corollary}\label{cds5w6} Consider any sequence $(A_k)_{k=1}^\infty$ with
$A_k\in\mathcal{L}(\mathcal{D}^{\odot  k},\mathcal{D})$, $k\geq 2$, and
$A_1=\mathbf 1$. Then there exists a unique sequence $(P^{(n)})_{n=0}^\infty$ of monic  polynomials on $\mathcal D'$ of binomial type that has the generating function~\eqref{Eq23} with $A(\xi)$ given by \eqref{ft7r}.
\end{corollary}

\begin{proof}
Define $A(\xi)\in\mathcal S(\mathcal D,\mathcal D)$ by formula \eqref{ft7r}. For each $\omega\in\mathcal D'$, define
$(\frac1{n!}\,P^{(n)}(\omega))_{n=0}^\infty\in\mathcal F(\mathcal D')$ by formula \eqref{Eq23}. It easily follows from Definitions~\ref{ye6eu0y87} and \ref{cxes6u54}  that $(P^{(n)})_{n=0}^\infty$ is a monic polynomial sequence on $\mathcal D'$. Furthermore, for $n\in\mathbb N$, in the representation~\eqref{ty65e48z} of $P^{(n)}(\omega)$, we obtain $U_{n,0}=0$ so that $P^{(n)}(0)=0$. Now the statement follows from Theorem~\ref{Theorem1}.
\end{proof}

We will now prove Theorem~\ref{Theorem1}.

\begin{proof}[Proof of $\mathrm{(BT1)}\Rightarrow \mathrm{(BT2)}$]
First, we note that, for any $\eta,\zeta\in\mathcal D'$,
\begin{equation}\label{u6t8}
E(\zeta)Q(\eta)1=Q(\eta)E(\zeta)1=0.\end{equation}
Next, using the the binomial identity \eqref{Eq2}, we get,  for all $\xi\in\mathcal D$ and $n\in\N$,
\begin{align}
Q(\eta)E(\zeta)\langle P^{(n)},\xi^{\otimes n}\rangle
&=\sum_{k=0}^n\binom{n}{k}\langle P^{(n-k)}(\zeta),\xi^{\otimes (n-k)}\rangle
Q(\eta)\langle P^{(k)},\xi^{\otimes k}\rangle\notag\\
&=\langle\eta,\xi\rangle\sum_{k=1}^n\binom{n}{k}k\langle P^{(n-k)}(\zeta),\xi^{\otimes (n-k)}\rangle
\langle P^{(k-1)},\xi^{\otimes (k-1)}\rangle\notag\\
&=n\langle\eta,\xi\rangle\sum_{k=1}^n\binom{n-1}{k-1}\langle P^{(n-k)}(\zeta),\xi^{\otimes (n-k)}\rangle
\langle P^{(k-1)},\xi^{\otimes (k-1)}\rangle\notag\\
&=n\langle\eta,\xi\rangle\sum_{k=0}^{n-1}\binom{n-1}{k}\langle P^{(n-k-1)}(\zeta),\xi^{\otimes (n-k-1)}\rangle
\langle P^{(k)},\xi^{\otimes k}\rangle\notag\\
&=n\langle\eta,\xi\rangle E(\zeta)\langle P^{(n-1)},\xi^{\otimes (n-1)}\rangle\notag\\
&=E(\zeta)Q(\eta)\langle P^{(n)},\xi^{\otimes n}\rangle.\label{v65r7}
\end{align}
By \eqref{u6t8}, \eqref{v65r7}, and Lemma \ref{Lemma1}, we get $E(\zeta)Q(\eta)P=Q(\eta)E(\zeta)P$ for all $P\in\mathcal P(\mathcal D')$.\end{proof}

In order to prove the implication $\mathrm{(BT2)}\Rightarrow \mathrm{(BT1)}$, we first need the following

\begin{proposition}[Polynomial expansion]\label{Lemma5}
Let $(P^{(n)})_{n=0}^\infty$ be a monic polynomial sequence on $\mathcal D'$ such that $P^{(n)}(0)=0$ for all $n\in\N$,  or,
equivalently, for $P^{(n)}$ being of the form \eqref{ty65e48z}, $U_{n,0}=0$. Let
$(Q(\zeta))_{\zeta\in\mathcal D'}$ be the corresponding  lowering operators.
Then, for each $P\in\mathcal{P}(\mathcal{D}')$, we have
\begin{equation}
P(\omega)=\sum_{k=0}^\infty\frac{1}{k!}
\big\langle P^{(k)}(\omega)(x_1,\ldots,x_k), (Q(x_1)\dotsm Q(x_k)P)(0)\big\rangle,\quad
\omega\in\mathcal{D}'.\label{Eq8}
\end{equation}
Here,   for $k=0$, we set
$Q(x_1)\dotsm Q(x_k)P:=P$.
\end{proposition}

\begin{proof}
For $x_1,\dots,x_k\in\R^d$, $k\in\N$,  $\xi\in\mathcal D$, and $n\in\N$, we have
\begin{equation}
Q(x_1)\dotsm Q(x_k)\langle P^{(n)},\xi^{\otimes n}\rangle=
(n)_k\xi(x_1)\dotsm\xi(x_k)\langle P^{(n-k)},\xi^{\otimes (n-k)}\rangle,
\label{Eq9}
\end{equation}
where $(n)_k:=n(n-1)\dotsm(n-k+1)$. Note that $(n)_k=0$ for $k>n$.
Hence, for  $k,n\in\N_0$, one finds
\[
\left(Q(x_1)\dotsm Q(x_k)\langle P^{(n)},\xi^{\otimes n}\rangle\right)(0)=
\delta_{k,n}n!\,\xi^{\otimes n}(x_1,\dots,x_n),
\]
where $\delta_{k,n}$ denotes the Kronecker symbol.  Thus,
\[
\langle P^{(n)}(\omega),\xi^{\otimes n}\rangle=
\sum_{k=0}^\infty\frac{1}{k!}\big\langle P^{(k)}(\omega)(x_1,\ldots,x_k),\left(Q(x_1)\dotsm Q(x_k)\langle P^{(n)},\xi^{\otimes n}\rangle\right)(0)\big\rangle.
\]
Hence, by Lemma \ref{Lemma1}, formula \eqref{Eq8} holds for a generic
$P\in\mathcal{P}(\mathcal{D}')$.
\end{proof}

\begin{proof}[Proof of $\mathrm{(BT2)}\Rightarrow \mathrm{(BT1)}$]
Let $\zeta\in\mathcal{D}'$, $\xi\in\mathcal D$, and $n\in\N$. An application of Proposition~\ref{Lemma5} to the polynomial
$P=E(\zeta)\langle P^{(n)},\xi^{\otimes n}\rangle$ yields
\begin{equation}\label{yei754wwq}
\langle P^{(n)}(\omega+\zeta),\xi^{\otimes n}\rangle
=\sum_{k=0}^\infty\frac{1}{k!}
\big\langle P^{(k)}(\omega)(x_1,\ldots,x_k), (Q(x_1)\dotsm Q(x_k)E(\zeta)\langle P^{(n)},\xi^{\otimes n}\rangle)(0)\big\rangle,
\end{equation}
and by (BT2) and \eqref{Eq9}, we have, for $k\in\mathbb N$,
\begin{align*}
\left(Q(x_1)\dotsm Q(x_k)E(\zeta)\langle P^{(n)},\xi^{\otimes n}\rangle\right)(0)&=\left(E(\zeta) Q(x_1)\dotsm Q(x_k)\langle P^{(n)},\xi^{\otimes n}\rangle\right)(0)\\
&=\left(Q(x_1)\dotsm Q(x_k)\langle P^{(n)},\xi^{\otimes n}\rangle\right)(\zeta)\\
&=(n)_k\zeta(x_1)\dotsm\zeta(x_k)\langle P^{(n-k)}(\zeta),\xi^{\otimes (n-k)}\rangle.
\end{align*}
Hence,
$$
\langle P^{(n)}(\omega+\zeta),\xi^{\otimes n}\rangle
=\sum_{k=0}^n\binom{n}{k}\langle P^{(k)}(\omega),\xi^{\otimes k}\rangle\langle P^{(n-k)}(\zeta),\xi^{\otimes (n-k)}\rangle.\qedhere
$$
\end{proof}

Thus, we have proved the equivalence of (BT1) and (BT2). To continue the proof of Theorem \ref{Theorem1}, we will need the following result.

\begin{theorem}[Operator expansion theorem]\label{Corollary1}
Let $(P^{(n)})_{n=0}^\infty$ be a monic polynomial sequence  on
$\mathcal{D}'$ of binomial type, and let $(Q(\zeta))_{\zeta\in\mathcal{D}'}$
 be
the corresponding lowering operators. A
linear operator $T$ acting on $\mathcal{P}(\mathcal{D}')$ is continuous and shift-invariant if and only if
 there is a $(G^{(k)})_{k=0}^\infty\in\mathcal{F}(\mathcal{D}')$ such that, for each $P\in\mathcal{P}(\mathcal{D}')$,
\begin{equation}
(TP)(\omega)=\sum_{k=0}^\infty\frac1{k!}\,\langle G^{(k)}(x_1,\ldots,x_k),(Q(x_1)\dotsm Q(x_k)P)(\omega)\rangle,\quad\omega\in\mathcal{D}'.\label{Eq12}
\end{equation}
In the latter case, for each $k\in\N_0$ and
$f^{(k)}\in\mathcal{D}^{\odot k}$,
\begin{equation}
\langle G^{(k)},f^{(k)}\rangle=\left(T\langle P^{(k)},f^{(k)}\rangle\right)(0).\label{Eq11}
\end{equation}
\end{theorem}

\begin{remark}
Below we will sometimes write formula \eqref{Eq12} in the form
$$T=\sum_{k=0}^\infty\frac1{k!}\,\big\langle G^{(k)}(x_1,\ldots,x_k),Q(x_1)\dotsm Q(x_k) \big\rangle.$$
\end{remark}

\begin{proof}[Proof of Theorem \ref{Corollary1}]
  Let $n\in\mathbb N$,  $\omega,\zeta\in\mathcal D'$, and $\xi\in\mathcal D$.  By \eqref{yei754wwq},  we have
\begin{equation}\label{t8p969p}
\big(E(\zeta)\langle P^{(n)},\xi^{\otimes n}\rangle\big)(\omega)=\sum_{k=0}^\infty\frac1{k!}\big\langle P^{(k)}(\omega)(x_1,\dots,x_k)
 ,(Q(x_1)\dotsm Q(x_k)\langle P^{(n)},\xi^{\otimes n}\rangle)(\zeta)\big\rangle.\end{equation}
  Note that formula \eqref{t8p969p} remains true when $n=0$. Hence, by Lemma \ref{Lemma1}, for each $P\in\mathcal P(\mathcal D')$,
\begin{equation}
 \big(E(\zeta)P\big)(\omega)=\sum_{k=0}^\infty\frac1{k!}\big\langle P^{(k)}(\omega)(x_1,\dots,x_k)
 ,(Q(x_1)\dotsm Q(x_k)P)(\zeta)\big\rangle.\label{gytt7r}
\end{equation}

 Assume $T\in\mathcal{L}(\mathcal{P}(\mathcal{D}'))$ is shift-invariant.
 Swapping $\omega$ and $\zeta$ in \eqref{gytt7r} and applying $T$ to this equality, we get,  for any $\omega,\zeta\in\mathcal D'$ and $P\in\mathcal P(\mathcal D')$,
\begin{equation}\label{guft7ur76r}
(TE(\omega)P)(\zeta)=
\sum_{k=0}^\infty\frac{1}{k!}
\left(T\big\langle P^{(k)}(\cdot)(x_1,\ldots,x_k), (Q(x_1)\dotsm Q(x_k)P)(\omega)\big\rangle\right)(\zeta).
\end{equation}
By shift-invariance, the left hand side of \eqref{guft7ur76r} is equal  to
 $(TP)(\omega+\zeta)$. In particular, this holds for
$\zeta=0$:
\begin{equation}\label{hjf}
(TP)(\omega)=\sum_{k=0}^\infty\frac{1}{k!}
\left(T\big\langle P^{(k)}(\cdot)(x_1,\ldots,x_k), (Q(x_1)\dotsm Q(x_k)P)(\omega)\big\rangle\right)(0).
\end{equation}
Let $G^{(k)}\in\mathcal{D}^{\prime\odot k}$, $k\in\N_0$, be defined by \eqref{Eq11}. Then \eqref{Eq12} follows from \eqref{hjf}.

Conversely, let $(G^{(k)})_{k=0}^\infty\in\mathcal{F}(\mathcal{D}')$  be fixed, and let $T$ be given by \eqref{Eq12}. As easily seen, $T\in\mathcal L(\mathcal P(\mathcal D'))$.
For each $P\in\mathcal{P}(\mathcal{D}')$ and $\zeta\in\mathcal D'$, we get from \eqref{Eq12} and  (BT2):
\begin{align*}
(TE(\zeta)P)(\omega)&=\sum_{k=0}^\infty\frac1{k!}\,\big\langle G^{(k)}(x_1,\ldots,x_k),(E(\zeta)Q(x_1)\dotsm Q(x_k)P)(\omega)\big\rangle\\
&=\sum_{k=0}^\infty\frac1{k!}\,\big\langle G^{(k)}(x_1,\ldots,x_k),(Q(x_1)\dotsm Q(x_k)P)(\omega+\zeta)\big\rangle\\
&=(TP)(\omega+\zeta)=(E(\zeta)TP)(\omega).
\end{align*}
Therefore, the operator $T$ is shift-invariant. Moreover, it easily follows from
\eqref{Eq12} and \eqref{Eq9} that \eqref{Eq11} holds.
\end{proof}

Note that the statement $\mathrm{(BT3)}\Rightarrow\mathrm{(BT2)}$ follows immediately from Theorem~\ref{Corollary1}.

\begin{proof}[Proof of $\mathrm{(BT2)}\Rightarrow \mathrm{(BT3)}$]
Let $\zeta\in\mathcal D'$. We apply Theorem \ref{Corollary1} to
the sequence of monomials and its  family of lowering operators, $(D(\eta))_{\eta\in\mathcal D'}$, and the shift-invariant operator $Q(\zeta)$. By using also formula \eqref{vgftydf7}, we obtain
\begin{equation}
Q(\zeta)  =\sum_{k=1}^\infty\frac{1}{k!}\big\langle G^{(k)}(\zeta,x_1,\dots,x_k),
D(x_1)\dotsm D(x_k)\big\rangle,\label{Eq15}
\end{equation}
where
\begin{equation}\label{gugtil}
\langle G^{(k)}(\zeta,x_1,\ldots,x_k),f^{(k)}(x_1,\ldots,x_k)\rangle=
\left(Q(\zeta)\langle\cdot^{\otimes k},f^{(k)}\rangle\right)(0)
=\langle\zeta,R_{1,k}f^{(k)}\rangle
\end{equation}
for all $k\in\N$ and $f^{(k)}\in\mathcal{D}^{\odot k}$. Here, we set $R_{1,1}:=\mathbf 1$, the identity operator on  $\mathcal D$.
For $k\in\mathbb N$, we denote $B_k:=R_{1,k}^*\in\mathcal{L}(\mathcal{D}',\mathcal{D}^{\prime\odot k})$. Note that $B_1=\mathbf 1$, the identity operator on  $\mathcal D'$. By \eqref{gugtil},
\begin{equation}
G^{(k)}(\zeta,\cdot)=B_k\zeta,\quad k\geq1.\label{Eq16}
\end{equation}
Formulas \eqref{Eq15}, \eqref{Eq16} imply (BT3).
\end{proof}

  Thus, we have proved the equivalence of (BT1), (BT2), and (BT3).

According to Theorem \ref{Corollary1}, under the conditions
assumed therein, there is a one-to-one correspondence between shift-invariant
operators $T$ and sequences $(\frac1{k!}\,G^{(k)})_{k=0}^\infty\in\mathcal{F}(\mathcal{D}')$.
We noted above that the space $\mathbb S(\mathcal P(\mathcal D'))$ of shift-invariant operators is an algebra under the product of operators, while $\mathcal{F}(\mathcal{D}')$ is a commutative algebra under the symmetric tensor product.

\begin{theorem}[The isomorphism theorem]\label{Theorem2}
Let $(P^{(n)})_{n=0}^\infty$ be a sequence of monic  polynomials on
$\mathcal{D}'$ of binomial type, and let $(Q(\zeta))_{\zeta\in\mathcal{D}'}$ be
the corresponding lowering operators. Then, the correspondence given by
Theorem \ref{Corollary1},
\[
\mathbb S(\mathcal P(\mathcal D'))\ni T\mapsto JT:=\left(\frac1{k!}\,G^{(k)}\right)_{k=0}^\infty\in\mathcal{F}(\mathcal{D}'),
\]
is an algebra isomorphism.
\end{theorem}
\begin{proof}
In view of Theorem \ref{Corollary1}, we only have to prove that, for
any $S,T\in\mathbb S(\mathcal P(\mathcal D'))$,
\begin{equation}\label{cftjrei6r8o}
J(ST)=JS\odot JT.
\end{equation}
Let
\[
JS=\left(\frac1{k!}\,F^{(k)}\right)_{k=0}^\infty,\quad JT= \left(\frac1{k!}\,G^{(k)}\right)_{k=0}^\infty.
\]
By Theorem \ref{Corollary1}, for all  $\xi\in\mathcal{D}$ and $\omega\in\mathcal D'$,
\begin{equation}\label{ghcdry6}
\left(T\langle P^{(n)},\xi^{\otimes n}\rangle\right)(\omega)
=\sum_{k=0}^n \binom nk\langle G^{(k)},\xi^{\otimes k}\rangle
\langle P^{(n-k)}(\omega),\xi^{\otimes (n-k)}\rangle,
\end{equation}
and a similar expression holds for $S$. Therefore,
\begin{align*}
&\left(ST\langle P^{(n)},\xi^{\otimes n}\rangle\right)(\omega)\\
&\quad=\sum_{k=0}^n \binom nk\langle G^{(k)},\xi^{\otimes k}\rangle
\left(S\langle P^{(n-k)},\xi^{\otimes (n-k)}\rangle\right)(\omega)\\
&\quad=\sum_{k=0}^n \binom nk\langle G^{(k)},\xi^{\otimes k}\rangle
\sum_{i=0}^{n-k}\binom {n-k}i\langle F^{(i)},\xi^{\otimes i}\rangle
\langle P^{(n-k-i)}(\omega),\xi^{\otimes (n-k-i)}\rangle\\
&\quad =\sum_{k=0}^n\sum_{i=0}^{n-k}\frac{n!}{k!\,i!\,(n-k-i)!}
\langle G^{(k)}\odot F^{(i)},\xi^{\otimes (k+i)}\rangle
\langle P^{(n-k-i)}(\omega),\xi^{\otimes (n-k-i)}\rangle\\
&\quad=\sum_{j=0}^n\binom nj \left\langle\sum_{k=0}^j \binom jk G^{(k)}\odot F^{(j-k)},
\xi^{\otimes j}\right\rangle\langle P^{(n-j)}(\omega),\xi^{\otimes (n-j)}\rangle.
\end{align*}
From here and \eqref{ghcdry6}, formula \eqref{cftjrei6r8o} follows.
\end{proof}

As an immediate consequence of Theorem \ref{Theorem2}, we conclude
\begin{corollary}\label{Corollary2}
Any  two shift-invariant operators  commute.
\end{corollary}

\begin{corollary}\label{65er5i458o55}
Let the conditions of Theorem \ref{Theorem2} be satisfied and let the operator $J:\mathbb S(\mathcal P(\mathcal D'))\to \mathcal F(\mathcal D')$ be defined as in that theorem. Define
$\mathcal J:\mathbb S(\mathcal P(\mathcal D'))\to\mathcal S(\mathcal D,\R)$
by
$\mathcal J:=\mathcal IJ$. Here $\mathcal I:\mathcal F(\mathcal D')\to\mathcal S(\mathcal D,\R)$ is defined by \eqref{ukt896p078}.
Then, for each $T\in\mathbb S(\mathcal P(\mathcal D'))$, we have
\begin{equation}
(\mathcal JT)(\xi)=\sum_{n=0}^\infty\frac{1}{n!}\left(T\langle P^{(n)},\xi^{\otimes n}\rangle\right)(0).\label{Eq18}
\end{equation}
Furthermore, $\mathcal J$ is an algebra isomorphism, i.e., for any $S,T\in\mathbb S(\mathcal P(\mathcal D'))$, we have
\begin{equation}\label{te56785}
\big(\mathcal J(ST)\big)(\xi)=(\mathcal JS)(\xi)(\mathcal JT)(\xi).\end{equation}
\end{corollary}

\begin{proof}
Formula \eqref{Eq18} follows Theorem \ref{Corollary1} and the definition of $\mathcal J$.  Formula \eqref{te56785} is a consequence of \eqref{ukt896p0781} and Theorem \ref{Theorem2}.
\end{proof}

\begin{corollary}\label{gd6ew6g}
Let $T\in\mathbb S(\mathcal P(\mathcal D'))$. The operator $T$ is invertible if and only if $T1\ne0$. Furthermore, if  $T1\ne0$, then $T^{-1}\in\mathbb S(\mathcal P(\mathcal D'))$.
\end{corollary}

\begin{proof} If $T1=0$, then the kernel of $T$ is not equal to $\{0\}$. Hence, $T$ is not invertible.

Assume $T1\ne0$. Let the isomorphism $J$ from Theorem \ref{Theorem2} be constructed through the monomials
 and the corresponding lowering operators
$D(\zeta)$, $\zeta\in\mathcal{D}'$. So formula \eqref{Eq18} becomes
\begin{equation}\label{bhuyfu}
(\mathcal JT)(\xi)=\sum_{n=0}^\infty\frac{1}{n!}\left(T\langle \cdot^{\otimes n},\xi^{\otimes n}\rangle\right)(0).\end{equation}
Since $T1\ne0$, formula \eqref{bhuyfu}, Corollary \ref{65er5i458o55}, and Proposition \ref{dr6e7i8o} imply the existence of an operator $S\in\mathbb S(\mathcal P(\mathcal D'))$ such that $ST=TS=\mathbf 1$. Hence, the operator $T$ is invertible and $T^{-1}=S\in \mathbb S(\mathcal P(\mathcal D'))$.
\end{proof}

\begin{proof}[Proof of $\mathrm{(BT3)}\Rightarrow \mathrm{(BT4)}$] Let the isomorphism $J$ from Theorem \ref{Theorem2} be constructed through the monomials
 and the corresponding lowering operators
$D(\zeta)$, $\zeta\in\mathcal{D}'$. By \eqref{bhuyfu},
\[
(\mathcal JD(\zeta))(\xi)=\langle\zeta,\xi\rangle.
\]
Thus, by Lemma \ref{Lemma4} and the isomorphism theorem,
\begin{equation}
(\mathcal JE(\zeta))(\xi)=\sum_{n=0}^\infty \frac{1}{n!}(\mathcal JD(\zeta)^n)(\xi)
=\sum_{n=0}^\infty \frac{1}{n!}\langle\zeta^{\otimes n},\xi^{\otimes n}\rangle
=\exp[\langle\zeta,\xi\rangle].\label{Eq20}
\end{equation}

Let $G^{(k)}\in\mathcal D'^{\odot k}$, $k\in\mathbb N$.
Then formula \eqref{bhuyfu} with
$$T=\big\langle G^{(k)}(x_1,\dots,x_k), D(x_1)\dotsm D(x_k)\big\rangle$$
yields
$$(\mathcal J\big\langle G^{(k)}(x_1,\dots,x_k), D(x_1)\dotsm D(x_k)\big\rangle)(\xi)
=\langle G^{(k)},\xi^{\otimes k}\rangle.$$
Therefore, condition (BT3) gives, for each $\zeta\in\mathcal D'$,
\begin{equation}\label{hgyut8tyt}
(\mathcal JQ(\zeta))(\xi)=\sum_{k=1}^\infty\frac{1}{k!}\langle B_k\zeta,\xi^{\otimes k}
\rangle=\sum_{k=1}^\infty\frac{1}{k!}\langle\zeta,R_{1,k}\xi^{\otimes k}\rangle.
\end{equation}
In the latter equality we used the fact that $R_{1,k}^*=B_k$, see the proof of $\mathrm{(BT2)}\Rightarrow \mathrm{(BT3)}$.
Choosing in \eqref{hgyut8tyt} $\zeta=\delta_x$, $x\in\R^d$, we obtain
\[
(\mathcal JQ(x))(\xi)=\sum_{k=1}^\infty\frac{1}{k!}\left(R_{1,k}\xi^{\otimes k}\right)(x),
\]
and, more generally, by Theorem \ref{Theorem2}, for any
$x_1,\ldots,x_k\in\R^d$, $k\in\N$,
\begin{equation}\label{iugit8t}
(\mathcal JQ(x_1)\dotsm Q(x_k))(\xi)=\prod_{i=1}^k\left(\sum_{n=1}^\infty\frac{1}{n!}\left(R_{1,n}\xi^{\otimes n}\right)(x_i)\right).
\end{equation}

By \eqref{bhuyfu} and \eqref{iugit8t}, we get, for each $G^{(k)}\in\mathcal D'^{\odot k}$, $k\in\mathbb N$,
\begin{align}
&\big(\mathcal J\langle G^{(k)}(x_1,\ldots,x_k),Q(x_1)\dotsm Q(x_k)\rangle\big)(\xi)\notag\\
&\quad=\big\langle G^{(k)}(x_1,\ldots,x_k),\mathcal J(Q(x_1)\dotsm Q(x_k))(\xi)\big\rangle\notag\\
&\quad=\left\langle G^{(k)}(x_1,\ldots,x_k),\prod_{i=1}^k\left(\sum_{n=1}^\infty\frac{1}{n!}\left(R_{1,n}\xi^{\otimes n}\right)(x_i)\right)\right\rangle\notag\\
&\quad=\left\langle G^{(k)},\left(\sum_{n=1}^\infty\frac{1}{n!}\,R_{1,n}\xi^{\otimes n}\right)^{\otimes k}\right\rangle.\label{ihgigit}
\end{align}

By \eqref{gytt7r} and  \eqref{ihgigit},
\begin{equation}
(\mathcal JE(\zeta))(\xi)=\sum_{k=0}^\infty\frac{1}{k!}
\left\langle P^{(k)}(\zeta), \left(\sum_{n=1}^\infty\frac{1}{n!}R_{1,n}\xi^{\otimes n}\right)^{\otimes k}\right\rangle.\label{Eq21}
\end{equation}
Formulas \eqref{Eq20} and \eqref{Eq21} imply
\begin{equation}
\sum_{k=0}^\infty\frac{1}{k!}
\left\langle P^{(k)}(\zeta), \left(\sum_{n=1}^\infty\frac{1}{n!}R_{1,n}\xi^{\otimes n}\right)^{\otimes k}\right\rangle=\exp[\langle\zeta,\xi\rangle].\label{Eq22}
\end{equation}

By Proposition \ref{rw4w4w5}, we find the compositional inverse $A(\xi)=\sum_{k=1}^\infty A_k\xi^{\otimes k}\in \mathcal S(\mathcal D,\mathcal D)$ of $\sum_{n=1}^\infty\frac{1}{n!}R_{1,n}\xi^{\otimes n}\in \mathcal S(\mathcal D,\mathcal D)$, and by  Remark~\ref{yrei4vvv}, we have $A_1=\mathbf 1$. 
Formula~\eqref{Eq23} now  follows from  \eqref{Eq22} and Proposition~\ref{yde57o5}.
\end{proof}

\begin{remark} It follows from the proof of Proposition \ref{rw4w4w5} that, for $k\ge 2$, the operators $A_k$  are given by the recurrence formula
\begin{equation}\label{buf7r8o}
A_k=-\sum_{n=2}^k \frac1{n!}\,R_{1,n}\sum_{\substack{(l_1,\dots,l_n)\in\mathbb N^n\\ l_1+\dots+ l_n=k}}A_{l_1}\odot\dots\odot A_{l_n}.\end{equation}

\end{remark}

\begin{proof}[Proof of $\mathrm{(BT4)}\Rightarrow \mathrm{(BT1)}$] By (BT4), we have, for any $\omega,\zeta\in\mathcal D'$,
\begin{align*}
\sum_{n=0}^\infty\frac{1}{n!}\langle P^{(n)}(\omega+\zeta),\xi^{\otimes n}\rangle&=\left(\sum_{n=0}^\infty\frac{1}{n!}\langle P^{(n)}(\omega),\xi^{\otimes n}\rangle\right)\left(\sum_{k=0}^\infty\frac{1}{k!}\langle P^{(k)}(\zeta),\xi^{\otimes k}\rangle\right)\\
&=\sum_{n=0}^\infty\frac{1}{n!}\left\langle\sum_{k=0}^n\binom{n}{k}P^{(k)}(\omega)\odot P^{(n-k)}(\zeta),\xi^{\otimes n}\right\rangle,
\end{align*}
which implies (BT1). This concludes the proof of Theorem~\ref{Theorem1}.
\end{proof}

\begin{definition}
Let $\left(Q(\zeta)\right)_{\zeta\in\mathcal{D}'}$ be a family of operators from $\mathcal L(\mathcal P(\mathcal D'))$. We  say that  $\left(Q(\zeta)\right)_{\zeta\in\mathcal{D}'}$ is a {\it family of delta
operators} if the following conditions are satisfied:

\begin{enumerate}
\item[(i)] Each $Q(\zeta)$ is shift-invariant;
\item[(ii)] For each $\zeta\in\mathcal{D}'$ and each $\xi\in\mathcal{D}$,
\begin{equation}\label{bhfu7frf7fr}
Q(\zeta)\langle\cdot,\xi\rangle=\langle\zeta,\xi\rangle;
\end{equation}
\item[(iii)] $Q(\zeta)$ linearly depends on $\zeta\in\mathcal{D}'$.
Furthermore, for each $k\geq 2$, the mapping
$\mathcal{D}'\ni\zeta\mapsto B_k\zeta\in\mathcal{D}'^{\odot k}$
defined by
\begin{equation}\label{hgvuft7fr}
\langle B_k\zeta,f^{(k)}\rangle:=\left(Q(\zeta)\langle\cdot^{\otimes k},f^{(k)}\rangle\right)(0),\quad f^{(k)}\in\mathcal{D}^{\odot k},
\end{equation}
belongs to $\mathcal{L}(\mathcal{D}',\mathcal{D}'^{\odot k})$.
\end{enumerate}
\end{definition}

It is a straightforward consequence of Theorem \ref{Theorem1} that, for any monic polynomial sequence
$(P^{(n)})_{n=0}^\infty$  of binomial type, the corresponding family
$\left(Q(\zeta)\right)_{\zeta\in\mathcal{D}'}$ of lowering operators is a family
of delta operators.

\begin{proposition}\label{Corollary3}
Let $\left(Q(\zeta)\right)_{\zeta\in\mathcal{D}'}$ be a family of delta operators.
Then, there exists a unique monic polynomial sequence $(P^{(n)})_{n=0}^\infty$ of binomial type for which
$\left(Q(\zeta)\right)_{\zeta\in\mathcal{D}'}$ is the family of lowering operators.
\end{proposition}

\begin{proof}
Let $\zeta,\omega\in\mathcal{D}'$ and $\xi\in \mathcal D$ . By shift-invariance of $Q(\zeta)$ and \eqref{bhfu7frf7fr}, we get
\begin{align*}
\langle\zeta,\xi\rangle&=E(\omega)\langle\zeta,\xi\rangle
=E(\omega)Q(\zeta)\langle\cdot,\xi\rangle
=Q(\zeta)E(\omega)\langle\cdot,\xi\rangle\\
&=Q(\zeta)\langle\cdot+\omega,\xi\rangle
=Q(\zeta)\langle\cdot,\xi\rangle+\langle\omega,\xi\rangle Q(\zeta)1=\langle\zeta,\xi\rangle+\langle\omega,\xi\rangle Q(\zeta)1,
\end{align*}
which implies that $Q(\zeta)1=0$. Hence, by Theorem \ref{Corollary1}, \eqref{bhfu7frf7fr}, and \eqref{hgvuft7fr}, we have
\[
Q(\zeta)=\sum_{k=1}^\infty\frac{1}{k!}\big\langle (B_k\zeta)(x_1,\ldots,x_k),D(x_1)\dotsm D(x_k)\rangle,
\]
with $B_1:=\mathbf 1$, the identity operator on  $\mathcal D'$.

Let $A_1:=\mathbf 1$ be the identity operator on  $\mathcal D$, and for $k\ge2$, let the operators $A_k\in\mathcal L(\mathcal D^{\odot k},\mathcal D)$ be defined by the recurrence formula \eqref{buf7r8o} with $R_{1,k}:=B_k^*$. Thus, $A(\xi)=\sum_{k=1}^\infty A_k\xi^{\otimes k}\in \mathcal S(\mathcal D,\mathcal D)$ is the compositional inverse of $\sum_{n=1}^\infty\frac{1}{n!}B_n^*\xi^{\otimes n}\in \mathcal S(\mathcal D,\mathcal D)$

Let $(P^{(n)})_{n=0}^\infty$ be the monic polynomial sequence  on $\mathcal D'$ of binomial type that has the generating function~\eqref{Eq23} with $A(\xi)$ given by \eqref{ft7r}, see Corollary~\ref{cds5w6}.
By Remark~\ref{bhgyutfu}, $(P^{(n)})_{n=0}^\infty$  is the unique required polynomial sequence.   \end{proof}

\begin{definition}
Let $\left(Q(\zeta)\right)_{\zeta\in\mathcal{D}'}$  be a family of delta operators.
The corresponding monic polynomial sequence $(P^{(n)})_{n=0}^\infty$ of binomial type given by Proposition \ref{Corollary3} is called the {\it basic sequence for $\left(Q(\zeta)\right)_{\zeta\in\mathcal{D}'}$}.
\end{definition}

\begin{proposition}\label{jiooyu}
 $\left(Q(\zeta)\right)_{\zeta\in\mathcal{D}'}$ is a family of delta operators if and only if
 there exists  a sequence $(B_k)_{k=1}^\infty$, with
$B_k\in\mathcal{L}(\mathcal{D}',\mathcal{D}'^{\odot  k})$, such that  $B_1=\mathbf 1$ and \eqref{tfer67} holds.
\end{proposition}

\begin{proof}The statement follows immediately from Theorem \ref{Theorem1} and Proposition~\ref{Corollary3}. \end{proof}

\section{Lifting of polynomials on $\R$ of binomial type}\label{Subsection3.2}

Let $(p_n)_{n=0}^\infty$ be a monic polynomial sequence  on $\R$ of binomial
type, and let $Q$ be its delta operator, that is, $Qp_n=np_{n-1}$ for each
$n\in\N_0$. According to the one-dimensional (classical) version of Theorem \ref{Theorem1}
(see e.g.~\cite{KRY}), $Q$ has a formal expansion
\begin{equation}\label{u8tr85t}
Q=\sum_{k=1}^\infty \frac{b_k}{k!}\,D^k=q(D),
\end{equation}
where $(b_k)_{k\in\N}$ is a sequence of real numbers such that $b_1=1$, $D$ is
the differentiation operator and
\begin{equation}\label{cdufti7wrt}
q(t):=\sum_{k=1}^\infty \frac{b_k}{k!}\,t^k
\end{equation}
is a formal power series in $t\in\R$. Furthermore,
\begin{equation}\label{ftuc}
\sum_{n=0}^\infty\frac{u^n}{n!}p_n(t)=\exp[ta(u)],
\end{equation}
where the formal power series in $u\in\R$,
\begin{equation}\label{uytr6d}
a(u)=\sum_{k=1}^\infty a_k u^k
\end{equation}
is the compositional inverse of $q$. In particular, $a_1=1$.
We will now lift the sequence of polynomials $(p_n)_{n=0}^\infty$ to a monic polynomial sequence  on $\mathcal{D}'$ of binomial type.

For each $k\in\mathbb N$, we define an operator $\mathbb D_k\in\mathcal L(\mathcal D^{\odot k},\mathcal D)$ by
\begin{equation}\label{sew46u}
(\mathbb D_kf^{(k)})(x):=f^{(k)}(x,\dots,x),\quad f^{(k)}\in\mathcal D^{\odot k},\ x\in\R^d\end{equation}
($\mathbb D_1$ being the identity operator on  $\mathcal D$).
The adjoint operator $\mathbb D_k^*\in\mathcal L(\mathcal D', \mathcal D'{}^{\odot k})$ satisfies
$$\langle \mathbb D_k^*\zeta,f^{(k)}\rangle=\langle \zeta(x),f^{(k)}(x,\dots,x)\rangle,\quad \zeta\in\mathcal D',\ f^{(k)}\in\mathcal D^{\odot k}.$$
In particular,
\begin{equation}\label{vytfyfr}
\langle \mathbb D_k^*\zeta,\xi^{\otimes k}\rangle=\langle\zeta,\xi^k\rangle, \quad \zeta\in\mathcal D',\ \xi\in\mathcal D.
\end{equation}
We now define an operator $B_k:\mathcal D'\to \mathcal D'{}^{\odot k}$ by
\begin{equation}\label{gutfr7tr7}
 B_k:=b_k \mathbb D_k^*,\quad k\in\mathbb N,\end{equation}
where the numbers $b_k$ are as in \eqref{u8tr85t}.
Let $(Q(\zeta))_{\zeta\in\mathcal D'}$ be the family of delta operators given by \eqref{tfer67}, see Proposition~\ref{jiooyu}.
By \eqref{vytfyfr} and \eqref{gutfr7tr7}, we then have
\begin{equation}\label{bgugfu8t}
Q(x)=\sum_{k=1}^\infty \frac{b_k}{k!}\,D^k(x)=q(D(x)),\quad x\in\R^d,
\end{equation}
and moreover,
\begin{equation}\label{gyutru8t8}
Q(\zeta)=\bigg\langle \zeta(x),\sum_{k=1}^\infty\frac{b_k}{k!}\,D^k(x)\bigg\rangle=\big\langle \zeta(x),q(D(x))\big\rangle.\end{equation}
Let $(P^{(n)})_{n=0}^\infty$ be the basic sequence for $(Q(\zeta))_{\zeta\in\mathcal D'}$. Thus, in view of \eqref{u8tr85t} and \eqref{bgugfu8t}, we may think of $(P^{(n)})_{n=0}^\infty$ as the {\it lifting of the monic polynomial sequence  $(p_n)_{n=0}^\infty$ of binomial type}.

As easily seen, the generating function of $(P^{(n)})_{n=0}^\infty$ is given by \eqref{Eq23} with the operators $A_k\in\mathcal L(\mathcal D^{\odot k}, \mathcal D)$ given by
$$ A_k=a_k\mathbb D_k,\quad k\in\mathbb N,$$
where the numbers $a_k$ are as in \eqref{uytr6d}.
Therefore, by \eqref{ft7r},
\begin{equation}\label{yte6790p0}
A(\xi)=\sum_{k=1}^\infty a_k\xi^k=a(\xi).\end{equation}
Hence,
\begin{equation}\label{bhgyut67rt}
\sum_{n=0}^\infty\frac{1}{n!}\langle P^{(n)}(\omega),\xi^{\otimes n}\rangle
=\exp\left[\left\langle\omega(x),\sum_{k=1}^\infty a_k\xi^k(x)\right\rangle\right]
=\exp\big[\langle\omega,a(\xi)\rangle\big].
\end{equation}
Thus, this generating function can be though of as the lifting of the generating function~\eqref{ftuc}.

Recall that a {\it set partition $\pi$ of a set $\mathcal X\ne\varnothing$}
 is an (unordered) collection of disjoint nonempty subsets of $\mathcal X$ whose union equals $\mathcal X$. We denote by $\mathfrak P(n)$ the collection of all  set partitions of $
 \mathcal X=\{1,2,\dots,n\}$.  For a set $B\subset\{1,\dots,n\}$, we denote by $|B|$ the cardinality of $B$.

\begin{proposition}\label{Proposition2}
Let $(P^{(n)})_{n=0}^\infty$ be the monic polynomial sequence on $\mathcal D'$  of binomial type that has the generating function  \eqref{bhgyut67rt}.
For $k\in\N$, denote $\alpha_k:=a_kk!$, so that
\begin{equation}\label{tydfrdrs5s}
a(u)=\sum_{k=1}^\infty\frac{\alpha_k}{k!}\,u^k.\end{equation}
Then, for any $n\in\mathbb N$,  $\omega\in\mathcal D'$, and $\xi\in\mathcal D$,
\begin{equation}
\langle P^{(n)}(\omega),\xi^{\otimes n}\rangle=\sum_{\pi\in\mathfrak P(n)}\prod_{B\in\pi}\alpha_{|B|}\langle \omega,\xi^{|B|}\rangle,
\label{yrd6iei}\end{equation}
or equivalently
\begin{multline}\label{6rr754}
P^{(n)}(\omega)=\sum_{k=1}^n
\sum_{\{B_1,B_2,\dots,B_k\}\in\mathfrak P(n)}
\alpha_{|B_1|}\alpha_{|B_1|}\dotsm \alpha_{|B_k|}\\
\times\big(\mathbb D^*_{|B_1|}\omega^{\otimes|B_1|}\big)\odot \big(\mathbb D^*_{|B_2|}\omega^{\otimes|B_2|}\big)\odot\dotsm \odot \big(\mathbb D^*_{|B_k|}\omega^{\otimes|B_k|}\big).
\end{multline}
\end{proposition}

\begin{proof}
It follows immediately from the form of the generating function \eqref{bhgyut67rt} that
\begin{align*}\langle P^{(n)}(\omega),\xi^{\otimes n}\rangle&=\sum_{k=1}^n\frac{n!}{k!}
\displaystyle\sum_{\substack{(i_1,\dots,i_k)\in\N^k\\ i_1+\dots +i_k=n}}a_{i_1}\dotsm a_{i_k}\langle\omega,\xi^{i_1}\rangle\dotsm \langle\omega,\xi^{i_k}\rangle\\
&=\sum_{\substack{(j_1,j_2\dots,j_n)\in\N_0^n\\j_1+2j_2+\dots+nj_n=n}}\frac{n!}{j_1!j_2!\dotsm j_n!\,(1!)^{j_1}(2!)^{j_2}\dotsm(n!)^{j_n}}\,\alpha_1^{j_1}\alpha_2^{j_2}\dotsm \alpha_n^{j_n}\notag\\
&\quad\times \langle \omega,\xi\rangle^{j_1} \langle \omega,\xi^2\rangle^{j_2}\dotsm \langle \omega,\xi^n\rangle^{j_n}.
\end{align*}
which implies \eqref{yrd6iei}, hence also \eqref{6rr754}.
\end{proof}

We denote by $\mathbb M(\R^d)$ the space  of all signed Radon measures on $\R^d$, i.e., the set of all signed measures $\eta$ on $(\R^d,\mathcal B(\R^d))$ such that $|\eta|(\Lambda)<\infty$ for all  $\Lambda\in\mathcal B_0(\R^d)$. Here $\mathcal B(\R^d)$ denotes the Borel $\sigma$-algebra on $\R^d$,   $\mathcal B_0(\R^d)$ denotes the collection of all bounded sets $\Lambda\in \mathcal B(\R^d)$, and  $|\eta|$ denotes the variation of $\eta\in \mathbb M(\R^d)$.

Further, let $\mathcal B_{\mathrm{sym}}((\R^d)^n)$ denote the sub-$\sigma$-algebra of $\mathcal B((\R^d)^n)$ that consists of all symmetric sets $\Delta\in \mathcal B((\R^d)^n)$, i.e., for each permutation $\sigma\in\mathfrak S(n)$, $\Delta$ is an invariant set for the mapping
$$(\R^d)^n\ni(x_1,\dots,x_n)\mapsto (x_{\sigma(1)},\dots,x_{\sigma(n)})\in(\R^d)^n.$$
We denote by $\mathbb M_{\mathrm{sym}}((\R^d)^n)$ the space of all signed Radon measures on $((\R^d)^n, \mathcal B_{\mathrm{sym}}((\R^d)^n))$.

\begin{corollary}\label{f675ew}
 Let $(P^{(n)})_{n=0}^\infty$ be the monic polynomial sequence on $\mathcal D'$  of binomial type that has the generating function  \eqref{bhgyut67rt}.  Then, for each $\eta\in \mathbb M(\R^d)$ and $n\in\mathbb N$, we have $P^{(n)}(\eta)\in\mathbb M_{\mathrm{sym}}((\R^d)^n)$. Furthermore, for each $\Lambda\in\mathcal B_0(\R^d)$,
\begin{equation}\label{vt7rei7549cd}
\big(P^{(n)}(\eta)\big)(\Lambda^n)=p_n(\eta(\Lambda)), \quad n\in\mathbb N,\end{equation}
where $(p_n)_{n=0}^\infty$ is the  polynomial sequence on $\R$ with generating function \eqref{ftuc}.
\end{corollary}

\begin{proof}
Let $\eta\in \mathbb M(\R^d)$. For each $j\in\mathbb N$, $\mathbb D_j^*\eta\in\mathbb M_{\mathrm{sym}}((\R^d)^j)$, since for each $f^{(j)}\in\mathcal D^{\odot j}$
$$\langle \mathbb D_j^*\eta,f^{(j)}\rangle=\int_{\R^d}f^{(j)}(x,\dots,x)\,d\eta(x).$$
Note that the measure $\mathbb D_j^*\eta$ is concentrated on the set
$$\{(x_1,x_2,\dots,x_j)\in(\R^d)^j\mid x_1=x_2=\dots=x_j\}.$$
By formula \eqref{6rr754}, we therefore get $P^{(n)}(\eta)\in \mathbb M_{\mathrm{sym}}((\R^d)^n)$.

Fix any $\Lambda\in\mathcal B_0(\R^d)$. Set $\xi:=u\chi_\Lambda$, where $u\in\R$ and $\chi_\Lambda$ denotes the indicator function of $\Lambda$. It easily follows from \eqref{bhgyut67rt} by an approximation  argument that
\begin{align}
\sum_{n=0}^\infty \frac{u^n}{n!}\big(P^{(n)}(\eta)\big)(\Lambda^n)&=\sum_{n=0}^\infty \frac1{n!}\,\langle P^{(n)}(\eta),\xi^{\otimes n}\rangle\notag\\
&=\exp\left[\int_{\R^d}a(\xi(x))\,d\eta(x)\right]\notag\\
&=\exp\left[\eta(\Lambda)a(u)\right].\label{dr6e65}
\end{align}
In formula \eqref{dr6e65}, $\langle P^{(n)}(\eta),\xi^{\otimes n}\rangle$ denotes the integral of the function $\xi^{\otimes n}$ with respect to the measure $P^{(n)}(\eta)$.
Formula \eqref{vt7rei7549cd} now follows from \eqref{ftuc} and \eqref{dr6e65}.
\end{proof}

The following proposition shows that the lifted polynomials $(P^{(n)})_{n=0}^\infty$  have an additional property of binomial type.

\begin{proposition}\label{uktro876p}
Let $(P^{(n)})_{n=0}^\infty$ be the monic polynomial sequence on $\mathcal D'$  of binomial type that has the generating function  \eqref{bhgyut67rt}. Let $\xi,\phi\in\mathcal D$ be such that \begin{equation}\label{gyut86}
\{x\in\R^d\mid\xi(x)\ne0\}\cap \{x\in\R^d\mid\phi(x)\ne0\}=\varnothing.\end{equation}
Then, for any $\omega\in\mathcal D'$ and $k,n\in\N$,
\begin{equation}\label{gyr7r5e}
\langle P^{(k+n)}(\omega),\xi^{\otimes k}\odot \phi^{\otimes n}\rangle=\langle P^{(k)}(\omega),\xi^{\otimes k}\rangle\langle P^{(n)}
(\omega),\phi^{\otimes n}\rangle.\end{equation}
Therefore, for each $n\in\mathbb N$,
\begin{equation}\label{uf7r6o8}
\langle P^{(n)}(\omega),(\xi+\phi)^{\otimes n}\rangle=\sum_{k=0}^n\binom nk \langle P^{(k)}(\omega),\xi^{\otimes k}\rangle \langle P^{(n-k)}
(\omega),\phi^{\otimes (n-k)}\rangle.\end{equation}
\end{proposition}

\begin{proof}
We have
\begin{align}
\sum_{n=0}^\infty \frac1{n!}\langle P^{(n)}(\omega),(\xi+\phi)^{\otimes n}\rangle&=\sum_{n=0}^\infty \sum_{k=0}^n \frac1{k!(n-k)!}\,\langle P^{(n)}(\omega),\xi^{\otimes k}\odot\phi^{\otimes(n-k)}\rangle\notag \\
&=\sum_{k=0}^\infty\sum_{n=0}^\infty \frac1{k!\, n!}
\langle P^{(n+k)}
(\omega),\phi^{\otimes n}\odot \xi^{\otimes k}\rangle,\label{gyygt8}
\end{align}
and by \eqref{bhgyut67rt} and \eqref{gyut86}
\begin{align}
\sum_{n=0}^\infty \frac1{n!}\langle P^{(n)}(\omega),(\xi+\phi)^{\otimes n}\rangle&=\exp\big[\langle\omega,a(\xi+\phi)\rangle\big]\notag\\
&=\exp\big[\langle\omega,a(\xi)\rangle\big] \exp\big[\langle\omega,a(\phi)\rangle\big]\notag\\
&=\sum_{k=0}^\infty\sum_{n=0}^\infty \frac1{k!\, n!}
\langle P^{(k)}(\omega),\xi^{\otimes k}\rangle\langle P^{(n)}
(\omega),\phi^{\otimes n}\rangle.\label{yur75r4dq}
\end{align}
Formulas \eqref{gyygt8}, \eqref{yur75r4dq} imply \eqref{gyr7r5e}, hence also \eqref{uf7r6o8}.
\end{proof}

We will now consider  examples of sequences of lifted polynomials of binomial type.

\subsection{Falling factorials on $\mathcal{D}'$}

The classical falling factorials is the sequence $(p_n)_{n=0}^\infty$ of monic polynomials on $\R$ of binomial type that are explicitly given by
$$p_n(t)=(t)_n:=t(t-1)(t-2)\dotsm (t-n+1).$$ The corresponding delta operator is $Q=e^D-1$, so that $Q$ is the difference operator $(Qp)(t)=p(t+1)-p(t)$. Here $p$ belongs to $\mathcal P(\R)$, the space of polynomials on $\R$. The generating function of the falling factorials is
$$\sum_{n=0}^\infty\frac{u^n}{n!}(t)_n=\exp[t\log(1+u)]=(1+u)^t.$$
One also defines an extension of the binomial coefficient,
\begin{equation}\label{iygt86t868}
\binom tn:=\frac1{n!}\,(t)_n=\frac{t(t-1)(t-2)\dotsm (t-n+1)}{n!},\end{equation}
which becomes the classical binomial coefficient for $t\in\N$, $t\ge n$.

Let us now consider the corresponding lifted sequence of polynomials, $(P^{(n)})_{n=0}^\infty$. We will call these polynomials the {\it falling factorials on $\mathcal D'$}. By analogy with the one-dimensional case, we will write $(\omega)_n:=P^{(n)}(\omega)$ for $\omega\in\mathcal D'$.

By \eqref{bgugfu8t}, $Q(x)=e^{D(x)}-1$. Hence, by Boole's formula,
\begin{equation}\label{ryde7r}
(Q(x)P)(\omega)=P(\omega+\delta_x)-P(\omega),\quad x\in\R^d,\ P\in\mathcal P(\mathcal D'),\end{equation}
and by \eqref{gyutru8t8},
$$(Q(\zeta)P)(\omega)=\big\langle \zeta(x), P(\omega+\delta_x)-P(\omega)\big\rangle,\quad \zeta\in\mathcal D', \ P\in\mathcal P(\mathcal D').$$
Further, by \eqref{bhgyut67rt}, the generating function  is given by
\begin{equation}\label{vtyfr57ir}
\sum_{n=0}^\infty\frac{1}{n!}\langle (\omega)_n,\xi^{\otimes n}\rangle=\exp\big[\langle\omega,\log(1+\xi)\rangle\big].\end{equation}

\begin{proposition}\label{Proposition1}
The falling factorials on $\mathcal{D}'$ have the following explicit form:
\begin{align}
&(\omega)_0=1;\notag\\
&(\omega)_1=\omega;\notag\\
&(\omega)_n(x_1,\ldots,x_n)=
\omega(x_1)(\omega(x_2)-\delta_{x_1}(x_2))\notag\\
&\qquad\ \times(\omega(x_3)-\delta_{x_1}(x_3)-\delta_{x_2}(x_3))\dotsm
(\omega(x_n)-\delta_{x_1}(x_n)-\dots -\delta_{x_{n-1}}(x_n))\label{fyd6e}
\end{align}
for $n\geq 2$.
\end{proposition}

\begin{proof} Let $((\omega)_n)_{n=0}^\infty$ denote the monic polynomial sequence  on $\mathcal D'$  defined by formula \eqref{fyd6e}. Note that, for $n\in\N$,  $(\omega)_n=0$ for $\omega=0$. It can be easily shown by induction that the  polynomials  $(\omega)_n$ satisfy the following recurrence relation:
\begin{align}
&(\omega)_0=1;\notag\\
&\langle (\omega)_{n+1},\xi^{\otimes (n+1)}\rangle
=\langle (\omega)_n\odot\omega,\xi^{\otimes (n+1)}\rangle-
n\langle (\omega)_n,\xi^2\odot\xi^{\otimes (n-1)}\rangle,\notag\\
&\qquad \xi\in\mathcal D,\ n\in\N_0.\label{tyr6i57e}
\end{align}
Furthermore, this recurrence relation uniquely determines the polynomials   $((\omega)_n)_{n=0}^\infty$. It suffices to prove that, for each $\omega\in\mathcal D'$, $n\in\mathbb N$, $\xi\in\mathcal D$, and $x\in\R^d$,
$$ \big(Q(x)\langle (\cdot)_n,\xi^{\otimes n}\rangle\big)(\omega)=n\xi(x)\langle (\omega)_{n-1},\xi^{\otimes (n-1)}\rangle,$$
or equivalently, by \eqref{ryde7r},
\begin{equation}
\langle (\omega+\delta_x)_n,\xi^{\otimes n}\rangle=\langle(\omega)_n,\xi^{\otimes n}\rangle+n\xi(x)\langle (\omega)_{n-1},\xi^{\otimes (n-1)}\rangle.
\label{yur75i}\end{equation}
We prove this formula by induction. It trivially holds for $n=1$. Assume that \eqref{yur75i} holds for $1,2,\dots,n$ and let us prove it for $n+1$. Using our assumption, the recurrence relation \eqref{tyr6i57e} and the polarization identity, we get
\begin{align*}
&\langle (\omega+\delta_x)_{n+1},\xi^{\otimes (n+1)}\rangle=
\langle (\omega+\delta_x)_n,\xi^{\otimes n}\rangle\langle\omega+\delta_x,\xi\rangle-n\langle (\omega+\delta_x)_n,\xi^2\odot\xi^{\otimes (n-1)}\rangle\\
&\quad=\big(\langle (\omega)_n,\xi^{\otimes n}\rangle+n\xi(x)\langle (\omega)_{n-1},\xi^{\otimes (n-1)}\rangle\big)\big(
\langle \omega,\xi\rangle +\xi(x)\big)\\
&\qquad-n\big( \langle (\omega)_n,\xi^2\odot\xi^{\otimes(n-1)}\rangle+\xi^2(x)\langle (\omega)_{n-1},\xi^{\otimes(n-1)}\rangle\\
&\qquad\qquad+(n-1)\xi(x)\langle (\omega)_{n-1},\xi^2\odot \xi^{\otimes(n-2)}\rangle\big)\\
&\quad=\big(\langle (\omega)_n\odot\omega,\xi^{\otimes (n+1)}\rangle-
n\langle (\omega)_n,\xi^2\odot\xi^{\otimes (n-1)}\rangle\big)+\xi(x)\langle (\omega)_n,\xi^{\otimes n}\rangle\\
&\qquad+n\xi(x)\big(
\langle (\omega)_{n-1}\odot\omega,\xi^{\otimes n}\rangle-(n-1)\langle (\omega)_{n-1},\xi^2\odot\xi^{\otimes(n-2)}\rangle\big)\\
&\quad=\langle (\omega)_{n+1},\xi^{\otimes (n+1)}\rangle+n\xi(x)\langle (\omega)_n,\xi^{\otimes n}\rangle. \qedhere
\end{align*}
\end{proof}

\begin{remark}
Note that the recurrence relation \eqref{tyr6i57e} satisfied by the polynomials on $\mathcal D'$ with generating function \eqref{vtyfr57ir} was already discussed in \cite{BKKL}.
\end{remark}

Since we have interpreted $(\omega)_n$ as a falling factorial on $\mathcal D'$, we naturally define {\it `$\omega$ choose $n$'} by $\binom\omega n:=\frac1{n!}(\omega)_n$, compare with \eqref{iygt86t868}.

We denote by $\Gamma$ the configuration space over $\R^d$, i.e., the space of all Radon measures $\gamma\in\mathbb M(\R^d)$ that are of the form $\gamma=\sum_{i=1}^\infty\delta_{x_i}$, where  $x_i\ne x_j$ if $i\ne j$. (We can obviously identify the configuration  $\gamma=\sum_{i=1}^\infty\delta_{x_i}$ with the (locally finite) set $\{x_i\}_{i\in\mathbb N}$.) The following result is immediate.

\begin{corollary}\label{uigt87}
For each $\gamma=\sum_{i=1}^\infty\delta_{x_i}$, formula \eqref{uytfr675e} holds.
\end{corollary}

\begin{remark} Polynomials $\binom\gamma n$ play a crucial role in the theory of point process (i.e., $\Gamma$-valued random variables), see e.g.\ \cite{KK}. More precisely, given a probability space $(\Xi,\mathcal F,\mathbb P)$ and a point process $\gamma:\Xi\to\Gamma$, the $n$th correlation measure of $\gamma$ is defined as the (unique) measure $\sigma^{(n)}$ on $\left((\R^d)^n,\mathcal B_{\mathrm{sym}}((\R^d)^n)\right)$ that satisfies
$$\mathbb E \bigg\langle\binom \gamma n, f^{(n)}\bigg\rangle=\int_{(\R^d)^n}
f^{(n)}(x_1,\dots,x_n)\,d\sigma^{(n)}(x_1,\dots,x_n)\quad\text{for all $f^{(n)}\in\mathcal D^{\odot n}$, $f^{(n)}\ge0$}.$$
Here $\mathbb E$ denotes the expectation with respect to the probability measure $\mathbb P$.
Under very mild conditions on the point process $\gamma$, the correlation measures $(\sigma^{(n)})_{n=1}^\infty$ uniquely identify the distribution of $\gamma$ on $\Gamma$.
In the case where each measure $\sigma^{(n)}$ is absolutely continuous with respect to the Lebesgue measure, one defines the $n$th correlation function of the point process $\gamma$, denote by $k^{(n)}(x_1,\dots,x_n)$, as follows:
$$d\sigma^{(n)}(x_1,\dots,x_n)=\frac1{n!}\,k^{(n)}(x_1,\dots,x_n)\,dx_1\dotsm dx_n,$$
or equivalently
$$\mathbb E \big\langle (\gamma)_n, f^{(n)}\big\rangle=\int_{(\R^d)^n}
f^{(n)}(x_1,\dots,x_n) k^{(n)}(x_1,\dots,x_n)\,\,dx_1\dotsm dx_n.$$
\end{remark}

\begin{corollary}\label{bvytfr76i}
For each $\Lambda\in\mathcal B_0(\R^d)$, $\gamma\in\Gamma$ and $n\in\mathbb N$, we have,
$$\binom\gamma n(\Lambda^n)=\binom{\gamma(\Lambda)}n.$$
\end{corollary}

\begin{proof}
The result is obvious by Corollary~\ref{uigt87}, or alternatively, by Corollary~\ref{f675ew}.
\end{proof}

\begin{remark}
In view of Corollaries \ref{uigt87} and \ref{bvytfr76i}, for the falling factorials on $\mathcal D'$, the set $\Gamma$ plays a role similar to that played by the set $\mathbb N$ for the falling factorials on $\R$.
\end{remark}

\subsection{Rising factorials on $\mathcal{D}'$}

The classical rising factorials is the sequence $(p_n)_{n=0}^\infty$ of monic polynomials on $\R$ of binomial type that are explicitly given by
$$p_n(t)=(t)^n:=t(t+1)(t+2)\dotsm (t+n-1).$$ The corresponding delta operator is $Q=1-e^{-D}$, so that $Q$ is the difference operator $(Qp)(t)=p(t)-p(t-1)$ for $p\in\mathcal P(\R)$. The generating function of the rising factorials is given by
$$\sum_{n=0}^\infty\frac{u^n}{n!}(t)^n=\exp[-t\log(1-u)]=(1-u)^{-t}.$$
One also has the following connection between the rising factorials and the falling factorials: $(t)^n=(-1)^n(-t)_n$.

Let us now consider the corresponding lifted sequence of polynomials, $(P^{(n)})_{n=0}^\infty$. We will call these polynomials the {\it rising factorials on $\mathcal D'$}, and we will write  $(\omega)^n:=P^{(n)}(\omega)$ for $\omega\in\mathcal D'$.

By \eqref{bgugfu8t}, $Q(x)=1-e^{-D(x)}$. Hence, by Boole's formula,
$$(Q(x)P)(\omega)=P(\omega)-P(\omega-\delta_x),\quad x\in\R^d,\ P\in\mathcal P(\mathcal D'),$$
and by \eqref{gyutru8t8},
$$(Q(\zeta)P)(\omega)=\big\langle \zeta(x), P(\omega)-P(\omega-\delta_x)\big\rangle,\quad \zeta\in\mathcal D', \ P\in\mathcal P(\mathcal D').$$
By \eqref{bhgyut67rt}, the generating function  is equal to
\begin{equation}\label{gygfyeaig}
\sum_{n=0}^\infty\frac{1}{n!}\langle (\omega)^n,\xi^{\otimes n}\rangle=\exp\big[\langle\omega,-\log(1-\xi)\rangle\big].\end{equation}

\begin{proposition}\label{iwqgf8yot}
We have $(\omega)^n=(-1)^n(-\omega)_n$ for all $n\in\N$, and the
 following explicit formulas hold:
\begin{align}
&(\omega)^0=1;\notag\\
&(\omega)^1=\omega;\notag\\
&(\omega)^n(x_1,\ldots,x_n)=
\omega(x_1)(\omega(x_2)+\delta_{x_1}(x_2))\notag\\
&\qquad\ \times(\omega(x_3)+\delta_{x_1}(x_3)+\delta_{x_2}(x_3))\dotsm
(\omega(x_n)+\delta_{x_1}(x_n)+\dots +\delta_{x_{n-1}}(x_n))\notag
\end{align}
for $n\geq 2$.
\end{proposition}

\begin{proof}
By \eqref{vtyfr57ir},
$$\sum_{n=0}^\infty\frac1{n!}\langle(-1)^n(-\omega)_n,\xi^{\otimes n}\rangle=\sum_{n=0}^\infty\frac1{n!}\langle (-\omega)_n,(-\xi)^{\otimes n}\rangle=\exp\big[\langle\omega,-\log(1-\xi)\rangle\big].$$
Hence, by \eqref{gygfyeaig}, $(\omega)^n=(-1)^n(-\omega)_n$ for all $\omega\in\mathcal D'$ and $n\in\mathbb N$. From this and Proposition \ref{Proposition1},  the statement follows.
\end{proof}

\subsection{Abel polynomials on $\mathcal D'$}
Let us fix a parameter $\alpha\in\R\setminus\{0\}$. The classical Abel polynomials on $\R$ corresponding to the parameter $\alpha$ is the monic polynomial sequence $(p_n)_{n=0}^\infty$  of binomial type that has the delta operator $Q=De^{\alpha D}$, i.e., $(Qp)(t)=p'(t+\alpha)$ for $p\in\mathcal P(\R)$. Thus, $Q=q(D)$, where
$ q(u)=ue^{\alpha u}$.
The generating function of the Abel polynomials is given by
$$\sum_{n=0}^\infty \frac{u^n}{n!}\,p_n(t)=\exp\big[t\alpha^{-1}W(\alpha u)\big],$$
where $W$ is the inverse function of $u\mapsto ue^u$ (around $0$),  the so-called
Lambert $W$-function.

Consider the corresponding lifted sequence of polynomials $(P^{(n)})_{n=0}^\infty$, the {\it Abel polynomials on $\mathcal D'$}. By Lemma \ref{Lemma4} and \eqref{bgugfu8t},
$$Q(x)=D(x)e^{\alpha D(x)}=D(x)E(\alpha \delta_x),\quad x\in\R^d,$$ and by \eqref{gyutru8t8},
$$Q(\zeta)P=\big\langle\zeta(x), D(x)P(\cdot+\alpha \delta_x)\big\rangle,\quad \zeta\in\mathcal D',\ P\in\mathcal P(\mathcal D').$$
By \eqref{bhgyut67rt}, the generating function  is equal to
\[
\sum_{n=0}^\infty\frac{1}{n!}\langle P^{(n)}(\omega),\xi^{\otimes n}\rangle
=\exp\left[\left\langle\omega,\alpha^{-1}W(\alpha\xi)\right\rangle\right].
\]

We have
 \[
\alpha^{-1}W(\alpha u)=\sum_{k=1}^\infty\frac{(-\alpha k)^{k-1}}{k!}u^k.
\]
Hence, by Proposition \ref{Proposition2},
 $$
 \langle P^{(n)}(\omega),\xi^{\otimes n}\rangle=\sum_{\pi\in \mathfrak P(n)}\prod_{B\in\pi}(-\alpha|B|)^{|B|-1}\langle\omega,\xi^{|B|}\rangle.$$

\subsection{Laguerre polynomials on $\mathcal D'$ of binomial type}
\label{hjgvfd7r6}

Let us recall that the (monic) Laguerre polynomials on $\R$ corresponding to a parameter $k\ge-1$, $(p_n^{[k]})_{n=0}^\infty$, have the generating function
\begin{equation}\label{hgfdyr77}
\sum_{n=0}^\infty \frac{u^n}{n!}\,p_n^{[k]}(t)=\exp\bigg[\frac{tu}{1+u}\bigg](1+u)^{-(k+1)}.\end{equation}
In particular, for the  parameter $k=-1$, the Laguerre polynomial sequence  $(p_n)_{n=0}^\infty:=(p_n^{[-1]})_{n=0}^\infty$ has the generating function
\begin{equation}\label{nhtfr76er}
\sum_{n=0}^\infty \frac{u^n}{n!}\,p_n(t)=\exp\bigg[\frac{tu}{1+u}\bigg].\end{equation}
Hence, the polynomial sequence $(p_n)_{n=0}^\infty$ is  of binomial type and its delta operator is
  $Q=q(D)$ with
$$ q(u)=\frac u{1-u}=\sum_{k=1}^\infty u^k.$$

The corresponding lifted sequence $(P^{(n)})_{n=0}^\infty$ will be called the {\it Laguerre polynomial sequence on $\mathcal D'$ of binomial type}. Thus, for each $x\in\R^d$, the  delta operator $Q(x)$ of
$(P^{(n)})_{n=0}^\infty$  is given by
\[
Q(x)=q(D(x))=\frac{D(x)}{1-D(x)}=\sum_{k=1}^\infty D(x)^k,
\]
and the corresponding generating function is
\begin{equation}\label{gdrydr6es6}
\sum_{n=0}^\infty\frac{1}{n!}\langle P^{(n)}(\omega),\xi^{\otimes n}\rangle
=\exp\left[\left\langle\omega,\frac{\xi}{1+\xi}\right\rangle\right].
\end{equation}

By Proposition \ref{Proposition2}, the polynomial $\langle P^{(n)}(\omega),\xi^{\otimes n}\rangle$ has representation \eqref{yrd6iei} with $\alpha_k=(-1)^{k+1}k!$\,. In view of the factor $k!$ in $\alpha_k$, we can also give the following combinatorial formula for  $\langle P^{(n)}(\omega),\xi^{\otimes n}\rangle$.

 Let $\mathfrak {W}(n)$ denote the collection of all sets
  $\beta=\{b_1,b_2,\dots,b_k\}$ such that each $b_i=(j_1,\dots,j_{l_i})$ is an element of $\{1,2,\dots,n\}^{l_i}$ with  $j_u\ne j_v$ if $u\ne v$, and each $j\in \{1,2,\dots,n\}$ is a coordinate of exactly one $b_i\in\beta$.
For $\beta\in \mathfrak {W}(n)$ and $b_i=(j_1,\dots,j_{l_i})\in\beta$, we denote $|b_i|:=l_i$.

By \eqref{yrd6iei}, we now get, for the Laguerre polynomials,
\begin{equation}\label{yre75o}
\langle P^{(n)}(\omega),\xi^{\otimes n}\rangle=\sum_{\beta\in \mathfrak {W}(n)}
\prod_{b\in\beta}\langle -\omega,(-\xi)^{|b|}\rangle.
\end{equation}

\section{Sheffer sequences}\label{Subsection3.3}

\begin{definition}
Let $\left(Q(\zeta)\right)_{\zeta\in\mathcal{D}'}$ be a family of delta operators. We say that  a monic polynomial sequence on $\mathcal D'$, $(S^{(n)})_{n=0}^\infty\,$,  is a {\it Sheffer sequence for the family of delta operators}
$\left(Q(\zeta)\right)_{\zeta\in\mathcal{D}'}$ if for each
$\zeta\in\mathcal{D}'$ and  $f^{(n)}\in\mathcal{D}^{\odot n}$,
$n\in\N$,
\begin{equation}
Q(\zeta)\langle S^{(n)},f^{(n)}\rangle=
\langle S^{(n-1)},\mathfrak A(\zeta)f^{(n)}\rangle,\label{Eq25}
\end{equation}
where $\mathfrak A(\zeta)$ is the annihilation operator, see Definition \ref{tee665i4ei}.
\end{definition}

Of course, any basic sequence for a family of delta operators is a
Sheffer sequence for that family of delta operators. 

\begin{theorem}\label{Theorem3}
Let $\left(Q(\zeta)\right)_{\zeta\in\mathcal{D}'}$ be a family of delta operators
and let $(P^{(n)})_{n=0}^\infty$ be its basic sequence, which has the generating
function  \eqref{Eq23}. Let $(S^{(n)})_{n=0}^\infty$ be a monic polynomial sequence  on $\mathcal{D}'$. Then the following
conditions are equivalent:

\begin{enumerate}
\item[{\rm (SS1)}] $(S^{(n)})_{n=0}^\infty$ is a Sheffer sequence for the family
$\left(Q(\zeta)\right)_{\zeta\in\mathcal{D}'}$.

\item[{\rm (SS2)}] There is a unique   operator $T\in\mathbb S(\mathcal P(\mathcal D'))$ such
that, for each $f^{(n)}\in\mathcal{D}^{\odot n}$, $n\in\N_0$,
\begin{equation}
T\langle S^{(n)},f^{(n)}\rangle=\langle P^{(n)},f^{(n)}\rangle.\label{Eq24}
\end{equation}

\item[{\rm (SS3)}] The sequence $(S^{(n)})_{n=0}^\infty$ has the generating function
\begin{equation}\label{ye7ir5p89t}
\sum_{n=0}^\infty\frac{1}{n!}\langle S^{(n)}(\omega),\xi^{\otimes n}\rangle
=\frac{\exp[\langle\omega,A(\xi)\rangle]}{\tau(A(\xi))},\quad\omega\in\mathcal{D}',\ \xi\in\mathcal{D},
\end{equation}
where $A(\xi)\in\mathcal S(\mathcal D,\mathcal D)$ is given by \eqref{ft7r} and $\tau\in\mathcal S(\mathcal D,\R)$ is such that $\tau(0)=1$.

\item[{\rm (SS4)}] For each $n\in\N$ and
$\omega,\zeta\in\mathcal{D}'$,
\begin{equation}\label{tyr65eei7}
S^{(n)}(\omega+\zeta)=\sum_{k=0}^n\binom{n}{k}S^{(k)}(\omega)\odot P^{(n-k)}(\zeta).
\end{equation}

\item[{\rm (SS5)}] There is a $(\rho^{(n)})_{n=0}^\infty\in\mathcal{F}(\mathcal{D}')$ with $\rho^{(0)}=1$
such that, for each $n\in\N$,
\[
S^{(n)}(\omega)=\sum_{k=0}^n\binom{n}{k}\rho^{(k)}\odot P^{(n-k)}(\omega),\quad\omega\in\mathcal{D}'.
\]
\end{enumerate}
\end{theorem}

\begin{remark}
For the meaning of the right hand side of formula \eqref{ye7ir5p89t}, see Proposition~\ref{dr6e7i8o} and Remark \ref{drs5w}.
\end{remark}


\begin{proof}[Proof of Theorem \ref{Theorem3}]

(SS1)$\Rightarrow $(SS2).
We define a linear operator $T:\mathcal P(\mathcal D')\to \mathcal P(\mathcal D')$ by formula \eqref{Eq24}. Since $(P^{(n)})_{n=0}^\infty$ and $(S^{(n)})_{n=0}^\infty$ are  monic polynomial sequences on $\mathcal{D}'$, we have $T\in\mathcal L(\mathcal P(\mathcal D'))$,  see formulas \eqref{gtdyf7ti} and \eqref{vgftydf7}. 
Thus, we  only have to prove that $T$ is shift-invariant. For this purpose, fix any $G^{(k)}\in\mathcal D'{}^{\odot k}$, $\xi\in\mathcal D$, and $k\in\mathbb N_0$.  By \eqref{Eq25} and  \eqref{Eq24},  we obtain
\begin{align}
&T\big\langle G^{(k)}(x_1,\dots,x_k), Q(x_1)\dotsm Q(x_k)
\langle S^{(n)},\xi^{\otimes n}\rangle
\big\rangle=(n)_k\langle G^{(k)},\xi^{\otimes k}\rangle T\langle S^{(n-k)},\xi^{\otimes (n-k)}\rangle\notag\\
&\quad =(n)_k\langle G^{(k)},\xi^{\otimes k}\rangle \langle P^{(n-k)},\xi^{\otimes (n-k)}\rangle\notag\\
&\quad = \big\langle G^{(k)}(x_1,\dots,x_k), Q(x_1)\dotsm Q(x_k)
\langle P^{(n)},\xi^{\otimes n}\rangle\big\rangle\notag\\
 &\quad = \big\langle G^{(k)}(x_1,\dots,x_k), Q(x_1)\dotsm Q(x_k)
T\langle S^{(n)},\xi^{\otimes n}\rangle\big\rangle.\notag
\end{align}
Therefore,
\[
T\langle G^{(k)}(x_1,\ldots,x_
k), Q(x_1)\dotsm Q(x_k)\rangle
=\langle G^{(k)}(x_1,\ldots,x_k), Q(x_1)\dotsm Q(x_k)\rangle T.
\]
Hence, by Theorem \ref{Corollary1}, $T$ is shift-invariant.

(SS2)$\Rightarrow$(SS1). By \eqref{Eq24} with $n=0$, we have $T1=1$. Hence, by Corollary \ref{gd6ew6g}, the operator $T$ is invertible and $T^{-1}\in\mathbb S(\mathcal P(\mathcal D'))$.  By Corollary \ref{Corollary2},  $T^{-1}$ commutes with each $Q(\zeta)$,
$\zeta\in\mathcal{D}'$. Therefore, for each $n\in\N$ and  $\xi\in\mathcal{D}$, we get
\begin{align*}
Q(\zeta)\langle S^{(n)},\xi^{\otimes n}\rangle
&=Q(\zeta)T^{-1}\langle P^{(n)},\xi^{\otimes n}\rangle
=T^{-1}Q(\zeta)\langle P^{(n)},\xi^{\otimes n}\rangle\\
&=T^{-1}(n\langle\zeta,\xi\rangle\langle P^{(n-1)},\xi^{\otimes (n-1)}\rangle)
=n\langle\zeta,\xi\rangle\langle S^{(n-1)},\xi^{\otimes (n-1)}\rangle.
\end{align*}

(SS2)$\Rightarrow$(SS3). For a fixed $\zeta\in\mathcal{D}'$, we apply Theorem~\ref{Corollary1} to the shift-invariant operator $E(\zeta)T^{-1}$. This gives
\[
(E(\zeta)T^{-1}P)(\omega)=
\sum_{k=0}^\infty\frac1{k!}
\big\langle G^{(k)}(x_1,\ldots,x_k), (Q(x_1)\dotsm Q(x_k)P)(\omega)\big\rangle,\quad P\in\mathcal P(\mathcal D'),
\]
where
\[
\langle G^{(k)},f^{(k)}\rangle=\left(E(\zeta)T^{-1}\langle P^{(k)},f^{(k)}\rangle\right)(0)=\left(E(\zeta)\langle S^{(k)},f^{(k)}\rangle\right)(0)=\langle S^{(k)}(\zeta),f^{(k)}\rangle.
\]
Thus,
\begin{equation}\label{te6u4}
(E(\zeta)T^{-1}P)(\omega)=
\sum_{k=0}^\infty\frac{1}{k!}
\big\langle S^{(k)}(\zeta) (x_1,\ldots,x_k), (Q(x_1)\dotsm Q(x_k)P)(\omega)\big\rangle,\quad P\in\mathcal P(\mathcal D').
\end{equation}

For the family of delta operators
$(D(\zeta))_{\zeta\in\mathcal{D}'}$ and its basic sequence
(monomials), consider  the isomorphism $\mathcal J$ defined in  Corollary \ref{65er5i458o55}. Hence, by \eqref{ihgigit} and \eqref{te6u4},
\begin{equation}\label{ye649450}
(\mathcal J(E(\zeta)T^{-1}))(\xi)=\sum_{k=0}^\infty\frac{1}{k!}
\left\langle S^{(k)}(\zeta),\left(\sum_{n=1}^\infty\frac{1}{n!}\left(R_{1,n}\xi^{\otimes n}\right)\right)^{\otimes k}\right\rangle.
\end{equation}
Furthermore, by \eqref{bhuyfu},
\begin{equation}
(\mathcal JT)(\xi)=1+\sum_{k=1}^\infty\frac{1}{k!}\langle\tau^{(k)},\xi^{\otimes k}\rangle=:\tau(\xi),\label{Eq26}
\end{equation}
where
\begin{equation}\label{yr7i8756r}
\langle\tau^{(k)},\xi^{\otimes k}\rangle:=\big(T\langle\cdot^{\otimes k},\xi^{\otimes k}\rangle\big)(0),\quad \xi\in\mathcal{D},\ k\in\N.
\end{equation}
(Note that the first term in \eqref{Eq26} is indeed equal to 1, because $T$
maps $1$ into $1$.) By \eqref{Eq20} and the isomorphism theorem,
\begin{equation}\label{yut8o50}
\exp[\langle\zeta,\xi\rangle]=(\mathcal JE(\zeta))(\xi)=\big(\mathcal J(E(\zeta)T^{-1})\big)(\xi)(\mathcal JT)(\xi).\end{equation}
By Proposition \ref{dr6e7i8o} (see also Remark \ref{drs5w}) and \eqref{ye649450}--\eqref{yut8o50}, we get
\begin{equation}\label{tyere49}
\sum_{k=0}^\infty\frac{1}{k!}
\left\langle S^{(k)}(\zeta),\left(\sum_{n=1}^\infty\frac{1}{n!}\left(R_{1,n}\xi^{\otimes n}\right)\right)^{\otimes k}\right\rangle=\frac{\exp[\langle\zeta,\xi\rangle]}{\tau(\xi)},
\end{equation}
compare with \eqref{Eq22}. We define operators $A_k\in\mathcal{L}(\mathcal{D}^{\odot k},\mathcal{D})$ by formula \eqref{buf7r8o} for $k\ge2$ and $ A_1:=\mathbf 1$. 
Thus, $A(\xi)=\sum_{k=1}^\infty A_k\xi^{\otimes k}\in \mathcal S(\mathcal D,\mathcal D)$ is the compositional inverse of $\sum_{n=1}^\infty\frac{1}{n!}R_{1,n}\xi^{\otimes n}\in \mathcal S(\mathcal D,\mathcal D)$. Now formula \eqref{tyere49} implies (SS3).

(SS3)$\Rightarrow$(SS2). For the family of delta operators
$(D(\zeta))_{\zeta\in\mathcal{D}'}$ and its basic sequence
(monomials), we construct the isomorphism $\mathcal J$.
We define an operator $T\in\mathbb S(\mathcal P(\mathcal D'))$ by $T:=\mathcal J^{-1}\tau$.
Since $\tau^{(0)}=1$, by Proposition \ref{dr6e7i8o} and Corollary \ref{65er5i458o55}, there exists an operator $S\in\mathbb S(\mathcal P(\mathcal D'))$ satisfying $ST=TS=\mathbf 1$. Hence, the operator $T$ is invertible and $T^{-1}=S$.

Since $T^{-1}1=1/\tau^{(0)}=1$,  for each $f^{(n)}\in\mathcal D^{\odot n}$, we obtain
\begin{equation}\label{zrw6}
\big(T^{-1}\langle\cdot^{\otimes n},f^{(n)}\rangle\big)(\omega)=
\langle \omega^{\otimes n},f^{(n)}\rangle+\sum_{i=0}^{n-1}\langle \omega^{\otimes i},g^{(i)}\rangle\end{equation}
for some $g^{(i)}\in\mathcal D^{\odot i}$, $i=0,1,\dots,n-1$.

Consider the linear operator
\begin{equation}\label{rds6ew6e}
\mathcal D^{\odot n}\ni f^{(n)}\mapsto  \tilde {\mathbb S}^{(n)}f^{(n)}:=T^{-1}\langle P^{(n)},f^{(n)}\rangle\in\mathcal P^{(n)}(\mathcal D').\end{equation}
Since $T^{-1}\in\mathcal L(\mathcal P(\mathcal D'))$ and satisfies \eqref{zrw6}, and  the linear operator
$$\mathcal D^{\odot n}\ni f^{(n)}\mapsto\langle P^{(n)},f^{(n)}\rangle\in\mathcal P^{(n)}(\mathcal D')$$
is continuous (see Lemma \ref{frt}), we conclude that $\tilde {\mathbb  S}^{(n)}\in\mathcal L(\mathcal D^{\odot n},\mathcal P^{(n)}(\mathcal D')) $. Hence, by Lemma \ref{frt} and  \eqref{zrw6}, there exists a monic polynomial sequence  $(\tilde S^{(n)})_{n=0}^\infty$ that satisfies
\begin{equation}\label{cxtarsra4a}
(\tilde {\mathbb S}^{(n)}f^{(n)})(\omega)=\langle \tilde S^{(n)}(\omega),f^{(n)}\rangle.\end{equation}
 By \eqref{rds6ew6e} and \eqref{cxtarsra4a}, we obtain
$$\langle P^{(n)},f^{(n)}\rangle=T\langle \tilde S^{(n)},f^{(n)}\rangle.$$

It follows from the proof of the implication (SS2)$\Rightarrow$(SS3) that the  monic polynomial sequence
$(\tilde S^{(n)})_{n=0}^\infty$ has the same generating function as
$(S^{(n)})_{n=0}^\infty$, so  they coincide. But this implies \eqref{Eq24}.

Thus, we have proved that the  conditions (SS1), (SS2), and (SS3) are equivalent.

(SS2)$\Rightarrow$(SS4). For $\omega,\zeta\in\mathcal D'$, $\xi\in\mathcal D$, and $n\in\N$ we have
\begin{align*}
\langle S^{(n)}(\omega+\zeta),\zeta^{\otimes n}\rangle&=\big(E(\zeta)\langle S^{(n)},\xi^{\otimes n}\rangle\big)(\omega)\\
&=\big(E(\zeta)T^{-1}\langle P^{(n)},\xi^{\otimes n}\rangle\big)(\omega)\\
&=\big(T^{-1}E(\zeta)\langle P^{(n)},\xi^{\otimes n}\rangle\big)(\omega)\\
&=\big(T^{-1}\langle P^{(n)}(\cdot+\zeta),\xi^{\otimes n}\rangle\big)(\omega)\\
&=\sum_{k=0}^n\binom nk \big(T^{-1}\langle P^{(k)},\xi^{\otimes k}\rangle)(\omega)\langle P^{(n-k)}(\zeta),\xi^{\otimes(n-k)}\rangle\\
&=\sum_{k=0}^n\binom nk \langle S^{(k)}(\omega),\xi^{\otimes k}\rangle\langle P^{(n-k)}(\zeta),\xi^{\otimes(n-k)}\rangle.
\end{align*}

(SS4)$\Rightarrow$(SS5). In formula \eqref{tyr65eei7}, swap $\omega$ and $\zeta$, set $\zeta=0$, and denote $\rho^{(n)}:=S^{(n)}(0)$. Note that $\rho^{(0)}=S^{(0)}(0)=1$.

(SS5)$\Rightarrow$(SS1). For each $\zeta\in\mathcal D'$, $\xi\in\mathcal D$, and $n\in\N$, we get from (SS5)
\begin{align*}
Q(\zeta)\langle S^{(n)},\xi^{\otimes n}\rangle&=\sum_{k=0}^n\binom nk \langle \rho^{(k)},\xi^{\otimes k}\rangle\, Q(\zeta)\langle P^{(n-k)},\xi^{\otimes(n-k)}\rangle\\
&=\sum_{k=0}^{n-1}\binom nk \langle \rho^{(k)},\xi^{\otimes k}\rangle (n-k)\langle\zeta,\xi \rangle\langle P^{(n-k-1)},\xi^{\otimes(n-k-1)}\rangle\\
&=n \langle\zeta,\xi \rangle \sum_{k=0}^{n-1}\binom {n-1}k \langle \rho^{(k)},\xi^{\otimes k}\rangle \langle P^{(n-k-1)},\xi^{\otimes(n-k-1)}\rangle\\
&= n \langle\zeta,\xi \rangle\langle S^{(n-1)},\xi^{\otimes(n-1)}\rangle.\qedhere
\end{align*}\end{proof}

\begin{corollary}\label{vghdf7r8o}
Let the conditions of Theorem \ref{Theorem3} be satisfied. Then  $(S^{(n)})_{n=0}^\infty$ is a Sheffer sequence for the family
$\left(Q(\zeta)\right)_{\zeta\in\mathcal{D}'}$  if and only if the sequence $(S^{(n)})_{n=0}^\infty$ has the generating function
\begin{equation}\label{yr8lio8ty}
\sum_{n=0}^\infty\frac{1}{n!}\langle S^{(n)}(\omega),\xi^{\otimes n}\rangle
=\exp\big[\langle\omega,A(\xi)\rangle- C(A(\xi))\big],\quad\omega\in\mathcal{D}',\ \xi\in\mathcal{D},
\end{equation}
where
$
A(\xi):=\sum_{k=1}^\infty A_k\xi^{\otimes k}\in\mathcal S(\mathcal D,\mathcal D)
$ is
as in \eqref{Eq23} and
$
C(\xi):=\sum_{k=1}^\infty\langle C^{(k)},\xi^{\otimes k}\rangle
\in\mathcal S(\mathcal D,\R)$ is such that $C(0)=0$.
\end{corollary}

\begin{proof}
We note that $\langle \omega,A(\xi)\rangle\in\mathcal S(\mathcal D,\R)$ is the composition of $\langle\omega,\xi\rangle\in\mathcal S(\mathcal D,\R)$ and $A(\xi)\in\mathcal S(\mathcal D,\mathcal D)$.
As easily seen, $\exp[\langle \omega,A(\xi)\rangle]\in\mathcal S(\mathcal D,\R)$ in formula \eqref{ye7ir5p89t} can be understood as the composition of $\exp t=\sum_{n=0}^\infty \frac{t^n}{n!}\in\mathcal S(\R,\R)$ and $\langle \omega,A(\xi)\rangle\in\mathcal S(\mathcal D,\R)$, see Definition~\ref{ydre7i5r8}.

Consider $\log(1+t)=\sum_{n=1}^\infty (-1)^{n+1}\,\frac{t^n}{n}\in\mathcal S(\R,\R)$. Define $C(\xi)\in\mathcal S(\mathcal D,\R)$ as the composition of $\log(1+t)\in \mathcal S(\R,\R)$ and $\tau(\xi)-1\in\mathcal S(\mathcal D,\R)$ (note that $\tau(0)-1=0$). Then $C(0)=0$. By Remark \ref{yjrde7ik87}, we obtain
$\tau(\xi)=\exp(C(\xi))$, the equality in $\mathcal S(\mathcal D,\R)$. Hence, using again Remark \ref{yjrde7ik87}, we conclude that formula \eqref{ye7ir5p89t} can be written as \eqref{yr8lio8ty}.
\end{proof}

\begin{remark}
According to Theorem \ref{Theorem3} and Corollary \ref{vghdf7r8o}, any Sheffer sequence
$(S^{(n)})_{n=0}^\infty$ is completely identified by a family of delta operators,
$\left(Q(\zeta)\right)_{\zeta\in\mathcal{D}'}$,  and a formal
 series $C\in\mathcal S(\mathcal D,\R)$  with $C(0)=0$.\end{remark}

\begin{corollary}\label{uftufr7t}
Under the conditions of Theorem \ref{Theorem3}, assume that
$(S^{(n)})_{n=0}^\infty$ is a Sheffer sequence. Let  $(\varkappa^{(n)})_{n=0}^\infty\in\mathcal{F}(\mathcal{D}')$ with $\varkappa^{(0)}=1$ satisfy
\begin{equation}\label{tye4849}
\tau(A(\xi))=\sum_{k=0}^\infty\frac{1}{k!}\langle\varkappa^{(k)},\xi^{\otimes k}\rangle.
\end{equation}
Then, for each $n\in\N$,
\begin{equation}\label{yd645u7}
P^{(n)}(\omega)=\sum_{k=0}^n\binom{n}{k}\varkappa^{(k)}\odot S^{(n-k)}(\omega),\quad\omega\in\mathcal{D}'.
\end{equation}
\end{corollary}

\begin{proof} By \eqref{Eq23}, (SS3) and \eqref{tye4849},
$$\sum_{n=0}^\infty\frac{1}{n!}\langle P^{(n)}(\omega),\xi^{\otimes n}\rangle=\left(\sum_{k=0}^\infty\frac{1}{k!}\langle\varkappa^{(k)},\xi^{\otimes k}\rangle\right)\left(\sum_{n=0}^\infty\frac{1}{n!}\langle S^{(n)}(\omega),\xi^{\otimes n}\rangle\right),$$
which implies \eqref{yd645u7}. \end{proof}

\begin{corollary}
Under the conditions of Theorem \ref{Theorem3}, assume that
$(S^{(n)})_{n=0}^\infty$ is a Sheffer sequence. Then,
$(S^{(n)})_{n=0}^\infty=(P^{(n)})_{n=0}^\infty$ if and only if $S^{(n)}(0)=0$ for each $n\in\N$.
\end{corollary}

\begin{proof} By (SS4), for each $n\in\N$ and
$\omega\in\mathcal{D}'$, we have
\[
S^{(n)}(\omega)=P^{(n)}(\omega)+\sum_{k=1}^n\binom{n}{k}S^{(k)}(0)\odot P^{(n-k)}(\omega).
\]
From here the statement follows.
\end{proof}

Let $\mathcal C(\mathcal D')$ denote the cylinder $\sigma$-algebra on $\mathcal D'$, i.e., the minimal $\sigma$-algebra on $\mathcal D'$ with respect to which each monomial $\langle\cdot,\xi\rangle$ ($\xi\in\mathcal D$) is measurable. A probability measure $\mu$ on $(\mathcal D',\mathcal C(\mathcal D'))$ is said to have finite moments if, for any $n\in\mathbb N$ and $f^{(n)}\in\mathcal D^{\odot n}$,
$$ \int_{\mathcal D'} |\langle\omega^{\otimes n}, f^{(n)}\rangle|\,d\mu(\omega)<\infty.$$

We now propose the following definition.

\begin{definition} Let $\mu$ be a probability measure on $(\mathcal D',\mathcal C(\mathcal D'))$ that has finite moments. Let $(S^{(n)})_{n=0}^\infty$ be a polynomial sequence on $\mathcal D'$. The polynomials $(S^{(n)})_{n=0}^\infty$  are said to be {\it orthogonal with respect to $\mu$} if
for any $m,n\in\mathbb N_0$, $m\ne n$, $f^{(m)}\in\mathcal D^{\odot m}$, and $g^{(n)}\in\mathcal D^{\odot n}$,
$$\int_{\mathcal D'}\langle S^{(m)}(\omega),f^{(m)}\rangle \langle S^{(n)}(\omega),g^{(n)}\rangle\,d\mu(\omega)=0.$$
\end{definition}

\begin{corollary}\label{gyure65e}Let $\mu$ be a probability measure on $(\mathcal D',\mathcal C(\mathcal D'))$ that has finite moments. Let $(S^{(n)})_{n=0}^\infty$ be a Sheffer sequence on $\mathcal D'$. Assume that $(S^{(n)})_{n=0}^\infty$ are orthogonal  with respect to $\mu$.
 Then the corresponding operator
 $T$ has the following representation:
\begin{equation}\label{dre6u3i7}
(TP)(\omega)=\int_{\mathcal D'}P(\omega+\zeta)\,d\mu(\zeta),\quad P\in\mathcal P(\mathcal D').\end{equation}
Furthermore, for $f^{(n)}\in\mathcal D^{\odot n}$, $n\in\mathbb N$, we have
\begin{equation}\label{cftres5yhte56}
\langle \tau^{(n)},f^{(n)}\rangle=\int_{\mathcal D'}\langle\omega^{\otimes n},f^{(n)}\rangle
\,d\mu(\omega).
\end{equation}
Here, $\tau(\xi)=1+\sum_{n=1}^\infty\frac1{n!}\,\langle\tau^{(n)},\xi^{\otimes n}\rangle\in\mathcal S(\mathcal D,\R)$ is as in \eqref{ye7ir5p89t}.
\end{corollary}

\begin{remark}
Formula \eqref{cftres5yhte56} states that each $\tau^{(n)}$ is the $n$th moment of the orthogonality measure $\mu$.
\end{remark}

\begin{proof}[Proof of Corollary \ref{gyure65e}] Note that formula \eqref{dre6u3i7} holds for $P=1$. Now, for each $n\in\mathbb N$ and $\xi\in\mathcal D$, we obtain  by (SS4):
\begin{align*}
\int_{\mathcal D'}\langle S^{(n)}(\omega+\zeta),\xi^{\otimes n}\rangle\,d\mu(\zeta)&=\sum_{k=0}^n \binom nk \langle P^{(n-k)}(\omega),\xi^{\otimes(n-k)}\rangle\int_{\mathcal D'}\langle S^{(k)}(\zeta),\xi^{\otimes k}\rangle\,d\mu(\zeta)\\
&=\langle  P^{(n)}(\omega),\xi^{\otimes n}\rangle
=\big(T\langle S^{(n)},\xi^{\otimes n}\rangle\big)(\omega).\qedhere
\end{align*}
Formula \eqref{cftres5yhte56} follows immediately from \eqref{yr7i8756r} and \eqref{dre6u3i7}.
\end{proof}

\begin{remark}
Note that, in the proof of Corollary \ref{gyure65e}, we only use the fact that
$\int_{\mathcal D'}\langle S^{(n)}(\omega),f^{(n)}\rangle\,d\mu(\omega)=0$ for all $n\in\mathbb N$.
\end{remark}

\begin{corollary}Assume that the conditions of Corollary \ref{gyure65e} are satisfied. Assume that there exists $\mathcal O\subset \mathcal D$, an  open neighborhood  of zero in $\mathcal D$, such that, for all $\xi\in\mathcal O$,
$$\int_{\mathcal D'} \exp\big[|\langle\omega,\xi\rangle|\big]\,d\mu(\omega)<\infty.$$
Then, for all $\xi\in\mathcal O$,
\begin{equation}\label{vut7}
\tau(\xi)=\int_{\mathcal D'}\exp\big[\langle\omega,\xi\rangle\big]\,d\mu(\omega),
\end{equation}
i.e., $\tau(\xi)$ is the Laplace transform of the measure $\mu$.
\end{corollary} 

\begin{proof}
Formula \eqref{vut7} follows from \eqref{cftres5yhte56} and the dominated convergence theorem. 
\end{proof}

Recall that a Sheffer sequence on $\R$ whose delta operator is the operator of differentiation is called an Appell sequence on $\R$.

\begin{definition}
Let $(S^{(n)})_{n=0}^\infty$ be a Sheffer sequence for the family of delta operators $(Q(\zeta))_{\zeta\in\mathcal D'}=(D(\zeta))_{\zeta\in\mathcal D'}$. Then we call $(S^{(n)})_{n=0}^\infty$ an {\it Appell sequence on $\mathcal D'$}.
\end{definition}

 By \eqref{Eq23}, $A(\xi)=\xi$ in the case of an Appell sequence. Hence, an Appell sequence $(S^{(n)})_{n=0}^\infty$ has the generating function
\[
\sum_{n=0}^\infty\frac{1}{n!}\langle S^{(n)}(\omega),\xi^{\otimes n}\rangle
=\exp\big[\langle\omega,\xi\rangle- C(\xi)\big],\quad\omega\in\mathcal{D}',\ \xi\in\mathcal{D}.
\]

\section{Lifting of Sheffer sequences on $\R$}\label{Subsection3.4}
We can extend the procedure described in Section \ref{Subsection3.2} to a
lifting of Sheffer sequences on $\R$. Let $(s_n)_{n=0}^\infty$ be a Sheffer sequence of monic polynomials on $\R$ for the delta operator $Q$, i.e.,  $Qs_n=ns_{n-1}$ for each $n\in\N_0$. Thus, $Q$ has representation \eqref{u8tr85t} and the polynomial sequence $(s_n)_{n=0}^\infty$ has the generating function
$$
\sum_{n=0}^\infty\frac{u^n}{n!}s_n(t)=\exp[ta(u)-c(a(u))],
$$
where  $a(u)\in\mathcal S(\R,\R)$ is given by \eqref{ftuc} (being the compositional inverse of the $q$ given by \eqref{cdufti7wrt}) and
$$ c(u)=\sum_{k=1}^\infty c_k u^k\in\mathcal S(\R,\R).$$

We now consider the family of delta operators, $(Q(\zeta))_{\zeta\in\mathcal D'}$, given
by \eqref{gyutru8t8}. Then $A(\xi)$ is given by \eqref{yte6790p0}. Furthermore, for $k\in\mathbb N$, we define $C^{(k)}\in\mathcal D'^{\odot k}$ by
$$\langle C^{(k)}, f^{(k)}\rangle:= c_k\langle \mathbb D_k f^{(k)}\rangle,\quad f^{(k)}\in\mathcal D^{\odot k}.$$
Here, the operator $\mathbb D_k$ is defined by \eqref{sew46u} and for $\xi\in\mathcal D$, we denote $\langle \xi\rangle:=\int_{\R^d}\xi(x)\,dx$. Thus, for $\xi\in\mathcal D$, we define
$$ C(\xi):=\sum_{k=1}^\infty \langle C^{(k)},\xi^{\otimes k}\rangle=\sum_{k=1}^\infty c_k\langle \xi^k\rangle=:\langle c(\xi)\rangle. $$
We now consider the Sheffer sequence $(S^{(n)})_{n=0}^\infty$ for the family of delta operators
$\left(Q(\zeta)\right)_{\zeta\in\mathcal{D}'}$ that has the generating function
\begin{equation}\label{cye56o}
\sum_{n=0}^\infty\frac{1}{n!}\langle S^{(n)}(\omega),\xi^{\otimes n}\rangle
=\exp\big[\langle\omega,a(\xi)\rangle- \langle c(a(\xi))\rangle\big],\quad\omega\in\mathcal{D}',\ \xi\in\mathcal{D},
\end{equation}
see \eqref{uytr6d} and Corollary \ref{vghdf7r8o}.
Thus, we may think of the Sheffer sequence $(S^{(n)})_{n=0}^\infty$ on $\mathcal D'$ as the lifting of the Sheffer sequence $(s_n)_{n=0}^\infty$ on $\R$.

\begin{proposition}\label{iut87o69}
Let $(S^{(n)})_{n=0}^\infty$ be a Sheffer sequence with generating function \eqref{cye56o}. Then, for each $n\in\mathbb N$, $\rho^{(n)}:=S^{(n)}(0)\in\mathcal D'{}^{\odot n}$ satisfies
\begin{equation}\label{vutr68l909}
\langle\rho^{(n)},\xi^{\otimes n}\rangle=\sum_{\pi\in\mathfrak P(n)}\prod_{B\in\pi}\lambda_{|B|}\langle\xi^{|B|}\rangle,\quad\xi\in\mathcal D,
\end{equation}
where $\lambda_k\in\R$ are defined by
 \begin{equation}\label{kytdk7i89}
\lambda(u):=-c(a(u))=\sum_{k=1}^\infty \frac{\lambda_k}{k!}\,u^k,\quad u\in\R.\end{equation}
Furthermore,
\begin{equation}\label{yur8o6rt9}
\langle S^{(n)}(\omega),\xi^{\otimes n}\rangle=\sum_{k=0}^n\binom{n}{k}\langle \rho^{(k)},\xi^{\otimes k}\rangle \langle P^{(n-k)}(\omega),\xi^{\otimes(n-k)}\rangle,\quad\omega\in\mathcal{D}',\xi\in\mathcal D,
\end{equation}
where $\langle P^{n}(\omega),\xi^{\otimes n}\rangle$ is given by formula \eqref{yrd6iei}.

\end{proposition}

\begin{proof} By \eqref{cye56o} with $\omega=0$ and \eqref{kytdk7i89}, we have
\begin{equation}\label{bu7t8p9}
\sum_{n=0}^\infty\frac1{n!}\,\langle \rho^{(n)},\xi^{\otimes n}\rangle=\exp\left[\sum_{k=1}^\infty\frac{\lambda_k}{k!}\langle\xi^k\rangle\right].\end{equation}
Using Fa\`a di Bruno's formula for the $n$th derivative of composition of functions, we deduce \eqref{vutr68l909} from \eqref{bu7t8p9}. Formula \eqref{yur8o6rt9} immediately follows from
(SS4) with $\omega=0$ (or (SS5)), \eqref{bhgyut67rt},
 and Proposition~\ref{Proposition2}.
\end{proof}

We can also write down the result of Proposition \ref{iut87o69} in the following form. By a {\it marked partition of the set $\{1,2,\dots,n\}$} we will mean a pair $(\pi,\mathfrak m_\pi)$ in which
 $\pi=\{B_1,B_2,\dots,B_k\}\in\mathfrak P(n)$ and  $\mathfrak m_\pi:\pi\to\{-,+\}$.
 (The value $\mathfrak m_\pi(B_i)\in\{-,+\}$ may be interpreted as the mark of the element $B_i$ of the partition $\pi$).
We will denote by $\mathfrak {MP}(n)$ the collection of all marked partitions of $\{1,2,\dots,n\}$.

\begin{corollary}\label{uyt8p}
Let $(S^{(n)})_{n=0}^\infty$ be a Sheffer sequence with generating function \eqref{cye56o}. Then, for each $n\in\mathbb N$, $\omega,\in\mathcal D'$, and $\xi\in\mathcal D$,
$$\langle S^{(n)}(\omega),\xi^{\otimes n}\rangle=\sum_{(\pi,\mathfrak m_\pi)\in\mathfrak{MP}(n)}
\left(\prod_{B\in\pi:\, \mathfrak m_\pi(B)=+}\alpha_{|B|}\langle \omega,\xi^{|B|}\rangle\right)\left(\prod_{B\in\pi:\, \mathfrak m_\pi(B)=-}\lambda_{|B|}\langle\xi^{|B|}\rangle\right),
$$
see formula \eqref{tydfrdrs5s} for the definition of $\alpha_k$.
\end{corollary}

\begin{proof}
Immediate from  Proposition \ref{iut87o69}.
\end{proof}

Using Proposition~\ref{iut87o69}, we can now immediately extend Corollary \ref{f675ew} to the case of a lifted Sheffer sequence.

\begin{corollary} \label{stsw5}
Let $(S^{(n)})_{n=0}^\infty$ be
a Sheffer sequence with generating function
\eqref{cye56o}. Then, for each $\eta\in \mathbb M(\R^d)$ and $n\in\mathbb N$, we have $S^{(n)}(\eta)\in\mathbb M_{\mathrm{sym}}((\R^d)^n)$. Furthermore, for each $\Lambda\in\mathcal B_0(\R^d)$, and $n\in\mathbb N$, we have
\begin{equation}\label{teei785}
\big(S^{(n)}(\eta))(\Lambda^n)=\tilde s_n(\eta(\Lambda)),
\end{equation}
 where $(\tilde s_n)_{n=0}^\infty$ is the Sheffer sequence on $\R$ with generating function $$
\sum_{n=0}^\infty\frac{u^n}{n!}\tilde s_n(t)=\exp\big[ta(u)-\operatorname{vol}(\Lambda)c(a(u))\big],$$
where $\operatorname{vol}(\Lambda):=\int_{\Lambda}dx$. In particular, $(\tilde s_n)_{n=0}^\infty=(s_n)_{n=0}^\infty$ if $\operatorname{vol}(\Lambda)=1.$
\end{corollary}

\begin{proposition}
The statement of Proposition \ref{uktro876p} remains true  for a Sheffer sequence $(S^{(n)})_{n=0}^\infty$ with generating function
\eqref{cye56o}.
\end{proposition}

\begin{proof} Analogously to the proof of Proposition \ref{uktro876p}, we note that, for any $\xi,\phi\in\mathcal D$ satisfying \eqref{gyut86},
\begin{align*}
&\exp\big[\langle\omega,a(\xi+\phi)\rangle- \langle c(a(\xi+\phi))\rangle\big]\\
&\quad= \exp\big[\langle\omega,a(\xi)\rangle- \langle c(a(\xi))\rangle\big] \exp\big[\langle\omega,a(\phi)\rangle- \langle c(a(\phi))\rangle\big].
\end{align*}
The rest of the proof is similar to that of Proposition \ref{uktro876p}.
\end{proof}

We will now consider  examples of lifted Sheffer sequences.

\subsection{Hermite polynomials on $\mathcal D'$}

The sequence of the Hermite polynomials on $\R$, $(s_n)_{n=0}^\infty$,  is the Appell sequence on $\R$ with $c(u)=\frac{u^2}{2}$. The Hermite polynomials are orthogonal with respect to the standard Gaussian (normal) distribution on $\R$. The  lifting of $(s_n)_{n=0}^\infty$ is the sequence
 of {\it Hermite polynomials on $\mathcal{D}'$}, $(S^{(n)})_{n=0}^\infty$, that has the
generating function
\begin{equation}\label{ytr8o67p}
\sum_{n=0}^\infty\frac{1}{n!}\langle S^{(n)}(\omega),\xi^{\otimes n}\rangle
=\exp\left[\langle\omega,\xi\rangle-\frac{\langle\xi^2\rangle}{2}\right]=\exp\left[\langle\omega,\xi\rangle-\frac12\|\xi\|^2_{L^2(\R^d,dx)}\right].
\end{equation}

\begin{remark}\label{Hermite}For each $\xi\in\mathcal D$ with $\|\xi\|_{L^2(\R^d,dx)}=1$, we get from \eqref{ytr8o67p} that
\begin{equation}\label{vfytrl8}
\langle S^{(n)}(\omega),\xi^{\otimes n}\rangle=s_n(\langle\omega,\xi\rangle).\end{equation}
We see that the Hermite polynomials, $S^{(n)}$, do not actually make use of the spatial structure of the underlying space, $\mathcal D'$, but  essentially use only the Hilbert space structure of $L^2(\R^d,dx)$. Formula \eqref{vfytrl8} is an exceptional property of the infinite-dimensional Hermite polynomials, compare with the general case discussed in Corollary \ref{stsw5}.
\end{remark}

Using either Proposition \ref{iut87o69} or Corollary \ref{uyt8p}, we easily get an explicit formula
\[
\langle S^{(n)}(\omega),\xi^{\otimes n}\rangle=\sum_{k=0}^{[\frac{n}{2}]}\binom{n}{2k}\frac{(2k)!}{k!\,2^k}(-\langle\xi^2\rangle)^k\langle\omega,\xi\rangle^{n-2k},
\]
where $[\frac{n}{2}]$ denotes the largest integer $\le \frac{n}{2}$.
(Note that $\frac{(2k)!}{k!\,2^k}$ is the number of all partitions $\pi\in\mathfrak P(2k)$ such that each set from the partition $\pi$ has precisely two elements.)

Let $\mu$ be the probability measure on $\mathcal D'$ that has Fourier transform
$$\int_{\mathcal D'}\exp[i\langle\omega,\xi\rangle]\,d\mu(\omega)=\exp\left[-\frac12\|\xi\|^2_{L^2(\R^d,dx)}\right], \quad \xi\in\mathcal D.$$
The measure $\mu$ is called the {\it Gaussian white noise measure}. The Hermite polynomials $(S^{(n)})_{n=0}^\infty$ are orthogonal with respect to $\mu$, and furthermore, for any $m,n\in\mathbb N$, $f^{(m)}\in\mathcal D^{\odot m}$, and $g^{(n)}\in\mathcal D^{\odot  n}$,
\begin{equation}\label{tyr78969p}
\int_{\mathcal D'}\langle S^{(m)}(\omega),f^{(m)}\rangle\langle S^{(n)}(\omega),g^{(n)}\rangle\,d\mu(\omega)=\delta_{m,n}n!(f^{(m)},g^{(n)})_{L^2(\R^d,dx)^{\odot n}},\end{equation}
see e.g.\  \cite{BK,HKPS}.

As pointed out in the Introduction, the infinite dimensional Hermite polynomials are well-known and play
a fundamental  role in Gaussian white noise analysis, see e.g.\ \cite{BK,HOUZ,HKPS,Obata} and the references therein. In white noise analysis,
one usually writes $:\omega^{\otimes n}:$ for  $S^{(n)}(\omega)$ and call it  the $n$th Wick power of $\omega$. In that context, the transformation $T$ given by formula \eqref{dre6u3i7} is known as the $C$-transform.

\subsection{Charlier polynomials on $\mathcal D'$}
The sequence of the Charlier polynomials on $\R$, $(s_n)_{n=0}^\infty$, is the Sheffer sequence with $a(u)=\log(1+u)$ and $c(u)=e^u-1$, so that $\lambda(u)=-u$. The Charlier polynomials are orthogonal with respect to the Poisson distribution corresponding to the intensity parameter 1. The  lifting of $(s_n)_{n=0}^\infty$ is the sequence
 of the {\it Charlier polynomials on $\mathcal{D}'$}, $(S^{(n)})_{n=0}^\infty$, that has the
generating function
\[
\sum_{n=0}^\infty\frac{1}{n!}\langle S^{(n)}(\omega),\xi^{\otimes n}\rangle
=\exp\left[\langle\omega,\log(1+\xi)\rangle-\langle\xi\rangle\right].
\]
Note that the corresponding binomial sequence is $((\omega)_n)_{n=0}^\infty$, the falling factorials on~$\mathcal D'$.

By Proposition \ref{iut87o69},
\begin{equation}\label{bhjgufydr7}
 \langle S^{(n)}(\omega),\xi^{\otimes n}\rangle=\sum_{k=0}^n\binom nk \langle-\xi\rangle^k\langle(\omega)_{n-k},\xi^{\otimes(n-k)}\rangle.\end{equation}
Furthermore, by Corollary \ref{uftufr7t}, we obtain
\begin{equation}\label{bgig8igt8ytg}
\langle(\omega)_{n},\xi^{\otimes n}\rangle=\sum_{k=0}^n \binom nk \langle\xi\rangle^k\langle S^{(n-k)}(\omega),\xi^{\otimes (n-k)}\rangle.
\end{equation}
Compare formulas \eqref{bhjgufydr7} and \eqref{bgig8igt8ytg} with Corollaries~2.9 and  2.10 in \cite{KKO2}, respectively. Note that the latter results were obtained only for $\omega$ from the configuration space $\Gamma$.

Let $\mu$ be the probability measure on $\mathcal D'$ that has Fourier transform
$$\int_{\mathcal D'}\exp[i\langle\omega,\xi\rangle]\,d\mu(\omega)=\exp\left[\int_{\R^d}(e^{i\xi(x)}-1)\,dx\right],\quad\xi\in\mathcal D.$$
The measure $\mu$ is concentrated on the configuration space $\Gamma$ and is called the {\it Poisson point process}, or the {\it Poisson white noise measure}. The Charlier polynomials $(S^{(n)})_{n=0}^\infty$ are orthogonal with respect to the Poisson point process $\mu$ and  formula \eqref{tyr78969p} holds true in this case.

The Charlier polynomials $(S^{(n)})_{n=0}^\infty$  play a fundamental role in Poisson analysis, see e.g.\ \cite{KKO,Kondratievetal,IK}. In this analysis, the transformation $T$ given by formula \eqref{dre6u3i7} is also known as the $C$-transform.

\subsection{Orthogonal Laguerre polynomials on $\mathcal D'$}

It follows from \eqref{hgfdyr77} that, for each parameter $k>-1$, the sequence of the Laguerre polynomials $(p_n^{[k]})_{n=0}^\infty$ on $\R$ corresponding to the parameter $k$ is a Sheffer sequence, whose corresponding binomial sequence is $(p_n)_{n=0}^\infty=(p_n^{[-1]})_{n=0}^\infty$, see \eqref{nhtfr76er}.
For each $k>-1$, the Laguerre polynomials $(p_n^{[k]})_{n=0}^\infty$  are orthogonal with respect to the gamma distribution
$$\frac1{\Gamma(k+1)}\,\chi_{(0,\infty)}(t)t^ke^{-t}\,dt.$$
In particular, for $k=0$, the Laguerre polynomials $(s_n)_{n=0}^\infty:=(p_n^{[0]})_{n=0}^\infty$ are orthogonal  with respect to the exponential distribution $\chi_{(0,\infty)}(t)e^{-t}\,dt$ on $\R$.
By \eqref{hgfdyr77},  $(s_n)_{n=0}^\infty$ is the Sheffer sequence with $a(u)=\frac{u}{1+u}$ and $c(u)=-\log(1-u)$, so that $\lambda(u)=-\log(1+u)$.

The lifting of $(s_n)_{n=0}^\infty$ is the sequence $(S^{(n)})_{n=0}^\infty$ of the {\it Laguerre polynomials on $\mathcal D'$} that has the generating function
\begin{equation}\label{cr7o8}
\sum_{n=0}^\infty\frac{1}{n!}\langle S^{(n)}(\omega),\xi^{\otimes n}\rangle
=\exp\left[\left\langle\omega,\frac{\xi}{1+\xi}\right\rangle-\langle\log(1+\xi)\rangle\right].
\end{equation}
Note that the corresponding polynomial sequence of binomial type is the Laguerre sequence $(P^{(n)})_{n=0}^\infty$ with generating function \eqref{gdrydr6es6}, see Subsection \ref{hjgvfd7r6}.
Analogously to formula \eqref{yre75o}, we will now present a combinatorial formula for $\langle S^{(n)}(\omega),\xi^{\otimes n}\rangle$.

 We can identify each permutation $\pi\in \mathfrak S(n)$ with  $c(\pi):=\{\nu_1,\dots,\nu_k\}$, the set of the cycles in $\pi$. For each cycle $\nu_i\in c(\pi)$, we denote by $|\nu_i|$ the length of the cycle $\nu_i$. We define
$$\mathfrak {MS}(n):=\big\{(\pi,\mathfrak m_\pi)\mid \pi\in \mathfrak S(n),\, \mathfrak m_\pi:c(\pi)\to\{+,-\} \big\},$$
compare with the definition of  $\mathfrak {MP}(n)$ above.

Note that, for a given subset of $\mathbb N$ that has $m$ elements, there are $(m-1)!$ cycles of length $m$ that contain the points from this set. Hence, by  Corollary \ref{uyt8p} and \eqref{cr7o8}, we  get:
$$\langle S^{(n)}(\omega),\xi^{\otimes n}\rangle=\!\!\sum_{(\pi,\mathfrak m_\pi)\in\mathfrak{MS}(n)}\!\!
\left(\prod_{\nu\in c(\pi):\, \mathfrak m_\pi(\nu)=+}|\nu|\big\langle -\omega,(-\xi)^{|\nu|}\big\rangle\right)\left(\prod_{\nu\in c(\pi):\, \mathfrak m_\pi(\nu)=-}\big\langle(-\xi)^{|\nu|}\big\rangle\right).
$$

 By Corollary \ref{stsw5}, formula \eqref{teei785} holds with  $(\tilde s_n)_{n=0}^\infty=(p_n^{[k]})_{n=0}^\infty$,
  the Laguerre polynomials on $\R$ corresponding to the parameter $k=\operatorname{vol}(\Lambda)-1>-1$.

 Let $\mu$ be the probability measure on $\mathcal D'$ that has the Laplace transform
 $$\int_{\mathcal D'}e^{-\langle\omega,\xi\rangle}\,d\mu(\omega)=
 \exp\left[-\int_{\R^d}\log(1+\xi(x))\,dx\right],\quad \xi\in\mathcal D,\ \xi>-1. $$
 The $\mu$ is called the {\it gamma measure}, or the {\it gamma completely random measure}. It is concentrated on the set of all (positive) discrete Radon measures $\eta=\sum_{i}s_i\delta_{x_i}\in\mathbb M(\R^d)$ with $s_i>0$ for all $i$. Note that, with $\mu$-probability one, the set of atoms of $\eta$, $\{x_i\}$, is dense in $\R^d$.

As follows from \cite{Kondratievetal,KL}, the Laguerre polynomials $(S^{(n)})_{n=0}^\infty$ are orthogonal with respect to the gamma measure $\mu$, and furthermore, for any $\xi,\psi\in\mathcal D$ and $m,n\in\N$,
$$
\int_{\mathcal D'}\langle S^{(m)}(\omega),\xi^{\otimes m}\rangle\langle S^{(n)}(\omega),\psi^{\otimes n}\rangle\,d\mu(\omega)=\delta_{m,n}n!\sum_{\pi\in \mathfrak S(n)}\prod_{\nu\in c(\pi)}\big\langle(\xi\psi)^{|\nu|}\big\rangle.$$

The Laguerre polynomials $(S^{(n)})_{n=0}^\infty$  play a fundamental role in gamma analysis, see e.g.\ \cite{Kondratievetal,KL,L1,L2}.

\subsection*{Acknowledgments}
The authors acknowledge the financial support of the SFB 701 ``Spectral
structures and topological methods in mathematics'', Bielefeld University.
MJO  was supported by the Portuguese national funds through
FCT---Funda{\c c}\~ao para a Ci\^encia e a Tecnologia, within the project
UID/MAT/04561/2013. YK and DF were supported by the European
Commission under the project STREVCOMS PIRSES-2013-612669.

\appendix
\renewcommand{\thesection}{A}
\section*{Appendix: Formal tensor power series}

We fix a general Gel'fand triple \eqref{fd6re7i}.
The following proposition is a direct consequence of formula \eqref{vut9p}.

\setcounter{theorem}{0}
\setcounter{equation}{0}

\begin{proposition}\label{dr6e7i8o}  Let  $F(\xi)=\sum_{n=0}^\infty\langle F^{(n)},\xi^{\otimes n}\rangle\in\mathcal S(\Phi,\R)$ be such that $F(0)=F^{(0)}\ne0$. Then there exists a unique $\sum_{n=0}^\infty\langle G^{(n)},\xi^{\otimes n}\rangle\in\mathcal S(\Phi,\R)$ such that
$$\left(\sum_{n=0}^\infty \langle F^{(n)},\xi^{\otimes n}\rangle\right)\left(\sum_{n=0}^\infty \langle G^{(n)},\xi^{\otimes n}\rangle\right)=1.$$
Explicitly, $G^{(0)}=1/F^{(0)}$ and for $n\ge1$, $G^{(n)}$ is recursively given by
$$G^{(n)}=-\frac1{F^{(0)}}\sum_{i=0}^{n-1} F^{(n-i)}\odot G^{(i)}.$$
We will denote
$$\left(\sum_{n=0}^\infty \langle F^{(n)},\xi^{\otimes n}\rangle\right)^{-1}:=\sum_{n=0}^\infty \langle G^{(n)},\xi^{\otimes n}\rangle.$$
\end{proposition}

\begin{remark}\label{drs5w}
It follows from  Proposition \ref{dr6e7i8o} that, for any $F(\xi),G(\xi)\in\mathcal S(\Phi,\R)$ with  $F(0)\ne0$, we obtain
$$\frac{G(\xi)}{F(\xi)}\in\mathcal S(\Phi,\R).$$
\end{remark}

\begin{definition}\label{ydre7i5r8}
Let  $R(t)=\sum_{n=0}^\infty r_nt^n\in \mathcal S(\R,\R)$.  For each $F(\xi)=\sum_{n=1}^\infty\langle F^{(n)},\xi^{\otimes n}\rangle\in\mathcal S(\Phi,\R)$ with $F(0)=0$, we define a {\it composition of $R$ and $F$}, denoted by $R\circ F(\xi)$ or
$$R(F(\xi))=\sum_{n=0}^\infty r_n\left(\sum_{k=1}^\infty
\langle F^{(k)},\xi^{\otimes k}\rangle
\right)^n,$$
as the formal series $G(\xi)=\sum_{n=0}^\infty \langle G^{(n)},\xi^{\otimes n}\rangle\in\mathcal S(\Phi,\R)$ with $G^{(0)}:=r_0$ and
$$G^{(n)}:=\sum_{m=1}^n\sum_{\substack{(k_1,\dots,k_m)\in\mathbb N^m\\ k_1+\dots+ k_m=n}}r_m\big(F^{(k_1)}\odot F^{(k_2)}\odot\dots\odot F^{(k_m)}\big),\quad n\in\mathbb N.$$
\end{definition}

\begin{definition}\label{cyr7t88p}
Let
$ A(\xi)=\sum_{n=1}^\infty  A_n \xi^{\otimes n},\  B(\xi)=\sum_{n=1}^\infty  B_n \xi^{\otimes n}\in\mathcal S(\Phi,\Phi)$. We define a {\it composition of $A$ and $B$}, denoted by $A\circ B(\xi)$ or 
$$ A(B(\xi))=\sum_{n=1}^\infty  A_n \left(\sum_{k=1}^\infty B_k\xi^{\otimes k}\right)^{\otimes n},$$
as the formal series $C(\xi)=\sum_{n=1}^\infty  C_n\xi^{\otimes n}\in\mathcal S(\Phi,\Phi)$ with
\begin{equation}\label{e7589}
 C_n:=
\sum_{m=1}^n\sum_{\substack{(k_1,\dots,k_m)\in\mathbb N^m\\ k_1+\dots+ k_m=n}}
 A_m\big(
 B_{k_1}\odot B_{k_2}\odot\dots\odot B_{k_m}
\big),\quad n\in\mathbb N.\end{equation}
Here, for  $(k_1,\dots,k_m)\in\mathbb N^m$ with $k_1+\dots+ k_m=n$, we denote
$$ B_{k_1}\odot B_{k_2}\odot\dots\odot B_{k_m}:=
\operatorname{Sym}_n( B_{k_1}\otimes B_{k_2}\otimes\dots\otimes B_{k_m}),$$
where $\operatorname{Sym}_n\in\mathcal L(\Phi^{\otimes n},\Phi^{\odot n})$ is the operator of symmetrization, see \eqref{cyrd6ue7r}.
\end{definition}

Similarly to Definitions \ref{ydre7i5r8} and  \ref{cyr7t88p}, we give the following

\begin{definition}\label{cxes6u54} Let  $F(\xi)=\sum_{n=0}^\infty\langle F^{(n)},\xi^{\otimes n}\rangle\in\mathcal S(\Phi,\R)$ and $ A(\xi)=\sum_{n=1}^\infty  A_n \xi^{\otimes n}\in\mathcal S(\Phi,\Phi)$. We define a composition of $F$ and $ A$, denoted by $F\circ A(\xi)$ or 
$$
F( A(\xi))=\sum_{n=0}^\infty\left\langle F^{(n)},\left(\sum_{k=1}^\infty  A_k \xi^{\otimes k}\right)^{\otimes n}\right\rangle,$$
as the  formal  series $G(\xi)=\sum_{n=0}^\infty\langle G^{(n)},\xi^{\otimes n}\rangle\in\mathcal S(\Phi,\R)$ with
 $G^{(0)}:=F^{(0)}$ and
$$ G^{(n)}:=\sum_{m=1}^n\sum_{\substack{(k_1,\dots,k_m)\in\mathbb N^m\\ k_1+\dots+ k_m=n}}( A_{k_1}^*\odot  A_{k_2}^*\odot\dots\odot  A_{k_m}^*)F^{(m)},\quad n\in\mathbb N.$$
Here $ A_k^*\in\mathcal{L}(\Phi',\Phi'{}^{\odot k})$ is the adjoint of $ A_k$.
\end{definition}

\begin{remark}\label{v65u45}
It follows from Definition \ref{cxes6u54} that
\begin{align}
&F( A(\xi))=F^{(0)}\notag\\
&\quad+\sum_{n=1}^\infty\left( \sum_{m=1}^n\sum_{\substack{(k_1,\dots,k_m)\in\mathbb N^m\\ k_1+\dots+ k_m=n}} \big\langle F^{(m)}, ( A_{k_1}\xi^{\otimes k_1})\odot ( A_{k_2}\xi^{\otimes k_2})\odot\dots\odot ( A_{k_m}\xi^{\otimes k_m})\big\rangle\right).\notag
\end{align}
\end{remark}

\begin{proposition}\label{yde57o5}
Let $F(\xi)\in \mathcal S(\Phi,\R)$ and let $ A(\xi), B(\xi)\in \mathcal S(\Phi,\Phi)$.
Then $$(F\circ A)\circ B(\xi)=F\circ (A\circ B)(\xi),$$
 the equality in $\mathcal S(\Phi,\R)$.
\end{proposition}

\begin{proof}
The proposition follows from  Definitions \ref{cyr7t88p} and \ref{cxes6u54}, see also Remark~\ref{v65u45}. We leave the details to the interested reader.
\end{proof}

\begin{remark}\label{yjrde7ik87}
In view of Proposition \ref{yde57o5}, we may just write $F\circ A\circ B(\xi)$.
As easily seen, a similar statement  also holds for the composition $S\circ R\circ F(\xi)\in\mathcal S(\Phi,\R)$, where $S,R\in\mathcal S(\R,\R)$ and $F\in\mathcal S(\Phi,\R)$, and for the composition $A\circ B\circ C(\xi)$, where $A,B,C\in\mathcal S(\Phi,\Phi)$.
\end{remark}

\begin{definition}\label{vtre7i}
 Let
$ A(\xi)\in\mathcal S(\Phi,\Phi)$. Then $ B(\xi)\in\mathcal S(\Phi,\Phi)$ is called the {\it compositional inverse of $A(\xi)$} if $A\circ B(\xi)=B\circ A(\xi)=\xi$. 
\end{definition}

\begin{remark}
Note that, if  $ B(\xi)\in\mathcal S(\Phi,\Phi)$ is the compositional inverse of $ A(\xi)\in\mathcal S(\Phi,\Phi)$, then $A(\xi)$ is the compositional inverse of $B(\xi)$.
\end{remark}

\begin{proposition}\label{rw4w4w5}
 Let
$ A(\xi)=\sum_{n=1}^\infty  A_n \xi^{\otimes n}\in\mathcal S(\Phi,\Phi)$ with  $ A_1\in\mathcal L(\Phi)$ being a homeo\-morphism.
Then there exists a unique compositional inverse  $ B(\xi)$ of $A(\xi)$.  \end{proposition}

\begin{proof} Let us first prove that there exists a unique  $B(\xi):=\sum_{n=1}^\infty  B_n \xi^{\otimes n}\in\mathcal S(\Phi,\Phi)$ such that $A\circ B(\xi)=\xi$. 
It follows from formula \eqref{e7589} that $C_1=A_1B_1$. Hence, for $C_1=\mathbf 1$, we must have $B_1=A_1^{-1}$. 
Now, by \eqref{e7589}  for $n\ge 2$, we get 
$$ C_n= A_1B_n+\sum_{m=2}^{n}\sum_{\substack{(k_1,\dots,k_m)\in\mathbb N^m\\ k_1+\dots+ k_m=n}}
 A_m\big(
 B_{k_1}\odot B_{k_2}\odot\dots\odot B_{k_m}
\big).$$
Hence,  we get $ C_n\mathbf=0$ for $n\ge2$ if and only if 
$$B_n=-A_1^{-1}\sum_{m=2}^{n}\sum_{\substack{(k_1,\dots,k_m)\in\mathbb N^m\\ k_1+\dots+ k_m=n}}
 A_m\big(
 B_{k_1}\odot B_{k_2}\odot\dots\odot B_{k_m}
\big).$$

Similarly, we prove that there exists a unique $\tilde B(\xi):=\sum_{n=1}^\infty  \tilde B_n \xi^{\otimes n}\in\mathcal S(\Phi,\Phi)$ such that $\tilde B\circ A(\xi)=\xi$. Here  $\tilde B_1=A_1^{-1}$ and for $n\ge2$,
$$ \tilde B_n=-\sum_{m=1}^{n-1}\sum_{\substack{(k_1,\dots,k_m)\in\mathbb N^m\\ k_1+\dots+ k_m=n}}
 \tilde B_m\big(
 A_{k_1}\odot A_{k_2}\odot\dots\odot A_{k_m}
\big)(A_1^{-1})^{\otimes n}.$$

Finally, we prove that $B(\xi)=\tilde B(\xi)$. Indeed, we get, using Remark \ref{yjrde7ik87},
$$ \tilde B(\xi)=\tilde B\circ (A\circ B)(\xi)=(\tilde B\circ A)\circ B(\xi)=B(\xi).$$
Hence, the proposition is proven.  
\end{proof}

\begin{remark}\label{yrei4vvv} It follows from Proposition \ref{rw4w4w5} and its proof that, if $ A(\xi)=\sum_{n=1}^\infty  A_n \xi^{\otimes n}\in\mathcal S(\Phi,\Phi)$ with  $A_1=\mathbf 1$, then its compositional inverse $ B(\xi):=\sum_{n=1}^\infty  B_n \xi^{\otimes n}$  exists and $B_1=\mathbf 1$. 
\end{remark}

\end{document}